\documentclass[hyperref]{article}
\usepackage{tikz}
\usepackage{graphicx} 
\usepackage{float} 
\usepackage{amssymb}
\usepackage{amsmath}
\usepackage{mathtools}
\usepackage[thmmarks,amsmath]{ntheorem} 
\usepackage{bm} 
\usepackage[colorlinks,linkcolor=cyan,citecolor=cyan]{hyperref}
\usepackage{extarrows}
\usepackage[hang,flushmargin]{footmisc} 
\usepackage[square,comma,sort&compress,numbers]{natbib} 
\usepackage{mathrsfs} 
\usepackage[font=footnotesize,skip=0pt,textfont=rm,labelfont=rm]{caption,subcaption}
\usepackage{booktabs} 
\usepackage{tocloft}
\usepackage{stmaryrd}
{
    \theoremstyle{nonumberplain}
    \theoremheaderfont{\bfseries}
    \theorembodyfont{\normalfont}
    \theoremsymbol{\mbox{$\Box$}}
    \newtheorem{proof}{Proof}
}
\usepackage{multirow}
\usepackage{theorem}
\newtheorem{theorem}{Theorem}[section]
\newtheorem{lemma}{Lemma}[section]
\newtheorem{definition}{Definition}[section]

{
    \theoremheaderfont{\bfseries}
    \theorembodyfont{\normalfont}
    
}

\usepackage{caption}
\newcommand{\myref}[1]{Eq.(\ref{#1}) }

\graphicspath{{figure/}}                                 
\usepackage[a4paper]{geometry}                           
\usepackage[english]{babel}
\usepackage{hyperref}
\addto\extrasenglish{%
}

\begin{document}
\newcommand{\topcaption}{%
\setlength{\abovecaptionskip}{0.cm}%
\setlength{\belowcaptionskip}{0.cm}%
\caption}

\title{\bf A multiscale method for inhomogeneous elastic problems with high contrast coefficients}
\date{}
\author{\sffamily Zhongqian Wang$^{1,*}$, Changqing Ye$^{1}$, Eric T. Chung$^1$\\
    {\sffamily\small $^1$ Department of Mathematics, The Chinese University of Hong Kong, Hong Kong Special Administrative Region }\\
}
\renewcommand{\thefootnote}{\fnsymbol{footnote}}
\footnotetext[1]{Corresponding author. }
\maketitle

{\noindent\small{\bf Abstract:}
     In this paper, we develop the constrained energy minimizing generalized multiscale finite element method (CEM-GMsFEM) with mixed boundary conditions (Dirichlet and Neumann) for the elasticity equations in high contrast media.  By a special treatment of mixed boundary conditions separately, and combining the construction of the relaxed and constraint version of the CEM-GMsFEM, we discover that the method offers some advantages such as the independence of the target region's contrast from precision, while the sizes of oversampling domains have a significant impact on numerical accuracy. Moreover, to our best knowledge, this is the first proof of the convergence of the CEM-GMsFEM with mixed boundary conditions for the elasticity equations given. Some numerical experiments are provided to demonstrate the method's performance. }

\vspace{1ex}
\section{Introduction}
The study of elastic inhomogeneous mixed boundary conditions is a major area of research in the field of inverse problems, which has wide and practical applications in many fields such as geophysics, oil exploration, remote sensing, ocean exploration, radar and sonar \cite{altawallbeh2018numerical,zhao2021robin,zhao2020staggered,betsch2020generic}. In the past decades, many researchers in different fields have been actively exploring computationally efficient methods for solving mixed boundary PDEs, for example, avoiding internal nodes or domain discretization to evaluate specific solutions \cite{wei2018boundary}, multiscale extensions and averaging techniques to approximate models with asymptotically small chemical patterns \cite{zhang2021effective,altmann2021numerical}, and inhomogeneous boundary treatments for fluid-related models in physics \cite{atlasiuk2019solvability,feireisl2018stationary}. A challenging problem that arises in this domain is the solid rheological properties of materials at the interface, such as the viscoelastic contact interface between transparent plexiglass components and supports used in aerospace and underwater vehicles, the creep problem at the contact interface between different metallic materials in high-temperature environments, and the mechanical behavior of materials with high contrast parameters for interfaces with complete bonding \cite{hauck2022multi,hauck2021super}. These traditional methods are no longer valid.

A different approach to the traditional problem is Multiscale Finite Element Methods (MsFEM) \cite{hou1997multiscale,allaire2005multiscale,chen2003mixed}. MsFEMs have been developed over the past decades by transforming microscale inhomogeneous information into macroscale parameters via homogenization or up-scaling, and then solving at the macroscale level, which mainly includes Representative Volume Element Methods (RVEMs) \cite{sun1996prediction}, Heterogeneous Multiscale Methods (HMMs) \cite{abdulle2012heterogeneous,engquist2007heterogeneous}. However, these methods are difficult to handle the grid-dependent elimination of the localization phenomenon. Generalized multiscale finite element method (GMsFEM) \cite{chung2014adaptive8} was then proposed to obtain an efficient eigenvalue approximation using a local-global model reduction technique, where the constructed global multiscale space can be applied repeatedly to get an efficient multiscale solution with reduced degrees of freedom at the macroscopic scale. Moreover, GMsFEM has been widely adopted in the field of equivalent parameter prediction, elastic wave propagation, acoustic analysis and gradient theory \cite{chung2017goal5}.  Regardless of implementation complexity,  the above methods have a similar issue: the potential high computing cost of handling inhomogeneous issues in intricate large-scale models.

In this paper, we aim to develop the constraint energy minimizing generalized multiscale Galerkin method (CEM-GMsFEM) to solve the complex elastic PDEs with mixed inhomogeneous boundary conditions. As far as we know, no previous research has investigated in this field. Although the bit-structural properties of common concrete \cite{hui2019application}, alloys, and other materials in engineering have obvious effects on macroscopic properties, these materials also have typical inhomogeneous and multiscale characteristics, and there are three main difficulties in using the FEMs to deal with these complex material boundary problems: \begin{itemize}
    \item In the physical sense, the intrinsic structure equations lack internal length parameters characterizing the microstructural features of the material, and heterogeneous theory cannot reasonably explain.
    \item In the numerical calculation, the spatial complexity of the microstructural information of heterogeneous materials can cause expensive computation time and the choice of minimum scale and minimum degrees of freedom is difficult to balance.
    \item In the validation of the computational method, when using oversampling techniques to construct multiscale basis functions, it is difficult to find a suitable model to improve the accuracy of the method.
\end{itemize}
We could use an improved CEM-GMsFEM to address the three limitations presented above \cite{chung2020computational}. Some works have demonstrated that the convergence of this method is independent of contrast, and that it decreases linearly with grid size for an appropriate choice of oversampling size \cite{vasilyeva2019constrained}. At the same time, CEM-GMsFEM is widely used in physical, geographical and environmental engineering \cite{Wang2022LocalMM,Wang2022Adg},  its online methods have been introduced to adaptively construct multiscale basis functions in certain regions to significantly reduce errors  \cite{chung2018fast223}. In the case of mixed boundary conditions, it is difficult to choose an adaptive method to handle multiple boundary cases simultaneously and by choosing a sufficient number of offline basis functions which can result in low errors at different contrast and grid sizes \cite{ye2021asymptotic,ye2020convergence,ye2022homogenization}. Using our recently proposed method and a special online basis construction for the oversampling regions, we show that the errors can be reduced sufficiently by an appropriate choice of the oversampling regions.

The paper is organized as follows. In \autoref{sub:1}, we introduce PDEs for mixed boundary conditions and the notations of grids. In \autoref{sub:2}, the CEM-GMsFEM used in this paper and the procedure for handling the mixed boundary conditions are presented. In \autoref{sub:3}, numerical analysis are given in \autoref{sub:3}. In \autoref{sub:4}, we conduct the conclusions.

\section{Problem formulation and fine grid approximation}\label{sub:1}
In this section, we give mathematical models of anisotropic elastic materials with inhomogeneous Dirichlet and  Neumann boundary conditions. In the following equations, the domain $\Omega \in \mathbb{R}^{d} \left(d =2,3\right)$ denotes an elastic body, and $\partial \Omega$ denotes its boundary with $\partial \Omega=\Gamma_{\rm{a}} \cup \Gamma_{\text{b}} $. The linear elasticity problem consists of finding the displacement $u$, such that:
\begin{equation}
    \left\{\begin{matrix}
 \text{-div}\left ( \sigma \left ( u \right ) \right )=f, & \rm{in}\  \Omega, \\
 u=h, & \rm{on} \ \Gamma_{\rm{a}}, \\
 \sigma \left ( u \right )\cdot n=g, & \rm{on}\ \Gamma_{\text{b}},  \\
 \sigma\left(u\right)=2 \mu \varepsilon\left(u\right)+\lambda \nabla \cdot u \mathcal{I},   &\rm{in}\  \Omega, \\
 \varepsilon \left ( u \right )=\frac{1}{2}\left [ \nabla u+\left ( \nabla u \right )^{T} \right ],      & \rm{in}\  \Omega,
\end{matrix}\right. \label{equinitial}
\end{equation}
where a body force $f$ is considered in the domain $\Omega$,

\qquad the surface force $g$ is on the Neumann boundary part,

\qquad $h$ is on the Dirichlet boundary part,

\qquad $\sigma $ and $\varepsilon$ are stress and strain tensor,

\qquad $n$ is the outward normal vector along $\Gamma_{\text{b}}$,

\qquad $\lambda$ and $\mu$ are the Lam\'{e} coefficients.

Then this inhomogeneous problem with Dirichlet conditions are incorporate through the decomposition $\tilde{u}=u-h$ such that $\tilde{u}=0$ on $\Gamma_{\rm{a}},$ i.e.,
$$\tilde{u} \in H_{\rm{a}}^{1}\left(\Omega\right):=\left\{v \in H^{1}\left(\Omega\right) \mid v=0 \text { on } \Gamma_{\rm{a}}\right\}.$$

We rewrite the problem (\ref{equinitial}) in a variation formulation: find $\tilde{u}\in H_{\rm{a}}^{1}\left(\Omega\right)$, such that
\begin{equation}
\int_{\Omega} \sigma \left ( \tilde{u} \right ) : \text{grad}\left ( v \right ) \,\textnormal{d} x=\int_{\Omega} f \cdot v \,\textnormal{d} x+\int_{\Gamma_{\text{b}}} g\cdot v \,\textnormal{d} s-\int_{\Omega} \sigma \left ( h\right ) : \text{grad}\left ( v \right ) \,\textnormal{d} x,\ \forall v \in H_{\rm{a}}^{1}\left ( \Omega \right ).\label{equ:galerkin}
\end{equation}

Let $\phi_{1},...,\phi_{n}$ be the basis set for $H_{\rm{a}}^{1}\left(\Omega\right),$ then $\tilde{u}_{h}$ satisfies
\begin{equation*}
    A_{h}\tilde{u}_{h}=L_{h},
\end{equation*}
where $A_{h}$ is symmetric, positive definite matrix with
\begin{equation*}
    A_{h,ij}=a\left(\phi_{i},\phi_{j}\right)=\int_{\Omega} \sigma \left ( \phi_{i}\right ) : \epsilon\left ( \phi_{j}\right ) \,\textnormal{d} x,
\end{equation*}
and $L_{h}$ is a vector with $i\text{-th}$ component
\begin{equation*}
    l\left(\phi_{i}\right)=\int_{\Omega} f \cdot \phi_{i} \,\textnormal{d} x+\int_{\Gamma_{\text{b}}} g \cdot \phi_{i} \,\textnormal{d} s-\int_{\Omega} \sigma \left ( h\right ) : \text{grad}\left ( \phi_{i} \right ) \,\textnormal{d} x.
\end{equation*}
Now we present GMsFEM. In this paper, we will develop and analyze the continuous Galerkin(CG) coupling after the construction of local basis functions. In essence, the CG coupling will need vertextrased local basis functions, then we give some mesh notations as follows.
Let $\mathcal{T}^{H}$ be a standard quadrilateralization of the domain $\Omega$, where we call the $\mathcal{T}^{H}$  coarse grid, $H>0$ being the coarse mesh size. Elements of $\mathcal{T}^{H}$ are called coarse lattice blocks. The set of all coarse gird edges is denoted by $\mathcal{E}^{H}$, and the set of all coarse gird nodes is denoted by $\mathcal{S}^{H}.$ We let $\mathcal{T}^{h}$ a fine mesh be the conformal refinement of the quadrilateral and $h>0$ the size of the fine mesh. We use $\mathcal{E}^{h}$ to denote the set of facets in $\mathcal{T}^{h}$ with $\mathcal{E}^{h}=\mathcal{E}_{\rm{a}}^{h} \cup \mathcal{E}_{\rm{b}}^{h}$.  In addition, we denote the number of fine grid nodes by $N_{H}$ and the number of fine grid blocks by $N$. If we consider a two-dimensional region of space $[0,1]\times [0,1]$ with $Nx$ and $Ny$ partitions in the $x$ and $y$ directions respectively, and $nx$ and $ny$ partitions on each fine grid, we have the following equation: $N_{H} = Nx*Ny*nx*ny,$ $ N= \left(Nx*nx+1\right)*\left(Ny*ny+1\right).$
We note that the refinement for the use of conformity is only intended to simplify the discussion of the method and is not a restriction on it.
As is shown in \autoref{picmesh}, we define $K_{i,m}$ as an oversampled domain on each $K_{i}\in \mathcal{T}^{H}:$
\begin{equation*}
    K_{i,m}=\bigcup \left \{ K_{j}\in \mathcal{T}^{H}|\overline{K}_{j}\cap  \overline{K}_{i}\neq 0 \right \}\cup \overline{K}_{i,m-1},
\end{equation*}
where $ \overline{K}_{i}$ is the closure of $K_{i},$ and the initial value $K_{i,0}=K_{i}$ for each element.

\begin{figure}[H]
\centering
\includegraphics[width=0.7\textwidth]{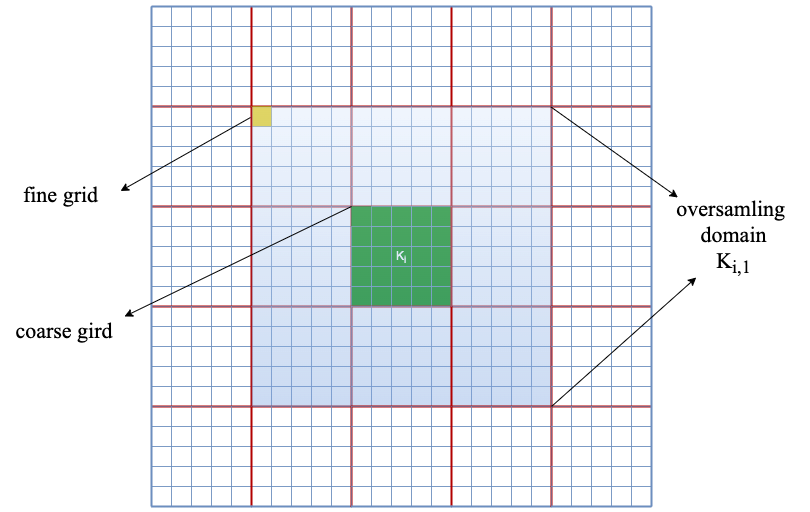}
\caption{Illustration of the oversampling domain}
\label{picmesh}
\end{figure}

\section{The construction of the CEM-GMsFEM basis function}\label{sub:2}
In this section, we describe the construction process of the CEM-GMsFEM, first constructing auxiliary basis functions, and then moving on to multiscale basis functions throughout the oversampling area utilizing constrained energy minimization.  We also provide an innovative approach to inhomogeneous elastic equations with natural and Dirichlet boundary conditions.
\subsection{Auxiliary basis function}
Let $V\left(K_{j}\right)$ be the snapshot space on each coarse grid block $K_{j}$, and we use the method of the spectral problem to solve the basis functions on $K_{j}$: find
$ \left(\theta _{j}^{i}, \phi_{j}^{i}\right) \in \mathbb{R} \times V\left(K_{j}\right)$ such that for all $v\in V\left(K_{j}\right),$
\begin{equation}
    a_{i}\left(\phi_{j}^{i}, v\right)=\theta _{j}^{i} s_{i}\left(\phi_{j}^{i}, v\right), \quad \forall v  \in V\left(K_{j}\right),
\end{equation}
where
\begin{equation}
    a_{i}\left(\phi_{j}^{i}, v\right)=\int_{K_{j}} \sigma \left (\phi_{j}^{i} \right ) : \epsilon\left ( v \right ) \,\textnormal{d} x,
\end{equation}
and
\begin{equation}
    s_{i}\left(\phi_{j}^{i}, v\right)=\int_{K_{j}} \tilde{\kappa} \phi_{j}^{i}\cdot v \,\textnormal{d} x
\end{equation}
with $\tilde{\kappa}=\sum_{i=1}^{N_{1}}\left(\lambda+2\mu\right)\nabla \chi_{i}\cdot \nabla \chi_{i},$ $N_{1}$ is the number of all coarse grids, $\{\nabla \chi_{i}\}_{i=1}^{N_1}$ is a set of partition of unity functions on the coarse grid.


Let the eigenvalues be in the following order:
$$\theta _{j}^{1}\leq \theta _{j}^{2}\leq \theta _{j}^{3}\leq ...\leq \theta _{j}^{i}\leq ...,$$
then the local auxiliary space $V_{\rm{aux}}\left(K_{j}\right)$ is defined by $g_j$ eigenvalue functions as following
\begin{equation}
    V_{\rm{aux}}\left(K_{j}\right)=\text{span}\left \{ \phi_{j}^{i}|1\leq j\leq g_j \right \}.
\end{equation}
We remark that $$V=\bigoplus_{j=1}^{N_{1}} V\left(K_{j}\right),\ V_{\mathrm{aux}}=\bigoplus_{j=1}^{N_{1}} V_{\mathrm{aux}}\left(K_{j}\right).$$
Let $s\left(\phi_{j}^{i}, v\right)=\sum_{j=1}^{N_1}s_{i}\left(\phi_{j}^{i}, v\right),$ we define that $\psi\in V$ is $\phi_{j}^{i}$-orthogonal if
$$s\left(\psi, \phi_{j}^{i}\right)=\left\{\begin{matrix}
1 ,& j^{\prime} \neq j,\\
 0 ,& j^{\prime} = j.
\end{matrix}\right.$$
In addition, we define a projection operator $\pi=\sum_{i=1}^{N_1}\pi_i$ from space $V$ to $V_{\rm{aux}}$ by
$$\pi_i\left(v\right)=\sum_{j=1}^{g_i} \frac{s_{i}\left(v, \phi_{j}^{i}\right)}{s_{i}\left(\phi_{j}^{i}, \phi_{j}^{i}\right)} \phi_{j}^{i}.$$

\subsection{ Offline multiscale basis functions}
In practice, many functions are discontinuous in the domain, so after constructing the auxiliary basis functions, we consider an oversampling method to extend these functions that are discontinuous in the domain onto the oversampling region $K_{j,m}$. For $\phi_{j}^{i} \in H_{\rm{a}}^{1}\left(K_{j,m}\right):=\left\{v \in H^{1}\left(K_{j,m}\right) \mid v=0 \text { on } \Gamma_{\rm{a}}\cap \partial K_{j,m}  \right\},$ we give the relaxed CEM-GMsFEM: find $\xi_{{\rm cem},j,m}^{\mathrm{I},i},$  such that
\begin{equation}
    \xi_{{\rm cem},j,m}^{\mathrm{I},i}=\operatorname{argmin}\left\{a\left(\xi, \xi\right)+s\left(\pi \xi-\phi_{j}^{i}, \pi \xi-\phi_{j}^{i}\right)| \xi \in H_{\rm{a}}^{1}\left(K_{j,m}\right)\right\}.\label{equrelaxed}
\end{equation}
We could also constrain the above problem:
\begin{equation}
    \xi_{{\rm cem},j,m}^{\mathrm{II},i}=\operatorname{argmin}\left\{a(\xi, \xi)| \xi \in H_{\rm{a}}^{1}\left(K_{j,m}\right), \xi \operatorname{is} \phi_{j}^{i} \operatorname{-orthogonal}\right\}.\label{equcons}
\end{equation}
In the following, we describe the constructed form of CEM-GMsFEM in two parts.

\textbf{Part.1. Relaxed CEM-GMsFEM}\\

Note the problem \myref{equrelaxed} is equivalent to the local problem: find $\xi_{{\rm cem},j,m}^{\mathrm{I},i}\in H_{\rm{a}}^{1}\left(K_{j,m}\right),$ such that
\begin{equation}
    a\left(\xi_{{\rm cem},j,m}^{\mathrm{I},i}, z_{1}\right)+s\left(\pi\left(\xi_{{\rm cem},j,m}^{\mathrm{I},i}\right), \pi\left(z_{1}\right)\right)=s\left(\phi_{j}^{i}, \pi\left(z_{1}\right)\right), \quad \forall z_{1} \in H_{\rm{a}}^{1}\left(K_{j,m}\right).
\end{equation}

Then, we obtain the following matrix representation from the above formulation
\begin{equation}
   \left(\mathcal{A}_{z_{1}}+\mathcal{M}_{z_{2}}\mathcal{M}_{z_{2}}^{T}\right) \mathbf{z_{1}}=\overline{\mathcal{M}_{z_{2}}},\label{matrix}
\end{equation}

where
$$\mathcal{A}_{z_{1}}:=\left(a\left(\xi_{i}, \xi_{j}\right)\right), \quad \mathcal{M}_{z_{2}}:=s\left(\left(q_{i},q_{i}\right)\right),$$

$ \mathbf{\theta }, \mathbf{z_{1}},\mathbf{z_{2}}$ are vectors of coefficients for the approximations, and $\overline{\mathcal{M}_{z_{2}}}$ is the internally reordered matrix

of  $\mathcal{M}_{z_{2}}.$

Finally we note that the global multiscale space is defined by
$$V_{{\rm cem}}:=\text{span}\left \{\xi_{{\rm cem},j,m}^{\mathrm{I},i}| 1\leq i \leq g_j, 1\leq j\leq N_1 \right \}.$$

The relaxed CEM-GMsFEM bases are defined by:
\begin{equation}
   \xi_{{\rm cem},j}^{\mathrm{I},i}=\operatorname{argmin}\left\{a\left(\xi, \xi\right)+s\left(\pi \xi-\phi_{{\rm cem},j}^{\mathrm{I},i}, \pi \xi-\phi_{j}^{i}\right)| \xi \in H_{\rm{a}}^{1}\left(\Omega\right) \right\},\label{equrelaxed2}
\end{equation}

which is equivalent to the following problem: find $\xi_{{\rm cem},j}^{\mathrm{I},i}\in V,$   such that
\begin{equation}
  a\left(\xi_{{\rm cem},j}^{\mathrm{I},i}, z_{1}\right)+s\left(\pi\left(\xi_{{\rm cem},j}^{\mathrm{I},i}\right), \pi\left(z_{1}\right)\right)=s\left(\phi_{j}^{i}, \pi\left(z_{1}\right)\right), \quad \forall z_{1} \in H_{\rm{a}}^{1}\left(\Omega\right).\label{equcemrelaxed}
\end{equation}

Furthermore, we define the global space of multiscale basis functions as $$V_{\rm{glo}}=\text{span}\left \{\xi_{{\rm cem},j}^{\mathrm{I},i}| 1\leq i \leq g_j, 1\leq j\leq N_1 \right \}. $$

\textbf{Part.2. Constrained CEM-GMsFEM}\\

Note the problem \myref{equcons} is equivalent to the local problem: find $\xi_{{\rm cem},j,m}^{\mathrm{II},i}\in H_{\rm{a}}^{1}\left(K_{j,m}\right),$

$\theta  \in V_{\rm{aux}}\left(K_{j,m}\right)$  such that
\begin{equation}
    \begin{aligned}
 a\left(\xi_{{\rm cem},j,m}^{\mathrm{II},i}, z_{1}\right)+s\left(z_{1}, \theta \right)&=0,  \quad \forall z_{1} \in H_{\rm{a}}^{1}\left(K_{j,m}\right), \\
s\left(\xi_{{\rm cem},j,m}^{\mathrm{II},i}-\phi_{j}^{i}, z_{2}\right)&=0,  \quad
 \forall z_{2} \in V_{\text {aux }}\left(K_{j,m}\right).
 \end{aligned}
\end{equation}

Then, we obtain the following matrix representation from the above formulation
\begin{equation}
    \begin{array}{r}
\mathcal{A}_{z_{1}}\mathbf{z_{1}}+\mathcal{M}_{z_{2}}\mathbf{\theta }=\mathbf{0}, \\
\mathcal{M}_{z_{2}}^{T} \mathbf{z_{2}}=\mathbf{I},
\end{array}
\end{equation}

where $\mathbf{\theta }, \mathbf{z_{1}},\mathbf{z_{2}}, \mathcal{A}_{z_{1}}, \mathcal{M}_{z_{2}}$ are as in $\myref{matrix}.$

Finally we note that the global multiscale space is defined by
$$V_{{\rm cem}}:=\text{span}\left \{\xi_{{\rm cem},j,m}^{\mathrm{II},i}| 1\leq i \leq g_j, 1\leq j\leq N_1 \right \}.$$

The constrained CEM-GMsFEM is defined by:
\begin{equation}
    \xi_{{\rm cem},j}^{\mathrm{II},i}=\operatorname{argmin}\left\{a(\xi, \xi)| \xi \in H_{\rm{a}}^1 \left(\Omega\right), \xi \operatorname{is} \phi_{j}^{i} \operatorname{-orthogonal}\right\},\label{equcons2}
\end{equation}

which is equivalent to the following problem: find $\xi_{{\rm cem},j}^{\mathrm{II},i}\in V,$ $\theta  \in V_{\rm{aux}}$  such that
\begin{equation}
    \begin{aligned}
 a\left(\xi_{{\rm cem},j}^{\mathrm{II},i}, z_{1}\right)+s\left(z_{1}, \theta\right )&=0,  \quad \forall z_{1} \in H_{\rm{a}}^{1}\left(\Omega \right), \\
s\left(\xi_{{\rm cem},j}^{\mathrm{II},i}-\phi_{j}^{i}, z_{2}\right)&=0,  \quad
 \forall z_{2} \in V_{\text {aux }}.
 \end{aligned}\label{equcemcons}
\end{equation}

Furthermore, we define the global space of multiscale basis functions as $$V_{\rm{glo}}=\text{span}\left \{\xi_{{\rm cem},j}^{\mathrm{II},i}| 1\leq i \leq g_j, 1\leq j\leq N_1 \right \}.$$

\subsection{Inhomogeneous BVPs multiscale method}
Above a thorough review of the CEM-GMsFEM, we provide two operators to cope with the boundary problem. We also divided our approach into two parts to describe. \\

\textbf{Part.1.The Relaxed multiscale method}\\

Firstly, for local auxiliary basis function space, we find $\mathcal{H}_{{\rm cem}}^{j,m} h \in H_{\rm{a}}^{1}\left(K_{j,m}\right),$ such that for all 

$z_{1}\in H_{\rm{a}}^{1}\left(K_{j,m}\right), $
\begin{equation}
a\left(\mathcal{H}_{{\rm cem}}^{j,m} h, z_{1}\right)+s\left(\pi \mathcal{H}_{{\rm cem}}^{j,m} h, \pi z_{1}\right) =\int_{K_{j}}  \sigma \left ( h\right ):\epsilon\left ( z_{1}\right ) \,\textnormal{d} x,
\end{equation}

Also, we find $\mathcal{G}_{{\rm cem}}^{j,m} h \in H_{\rm{a}}^{1}\left(K_{j,m}\right),$ such that for all $z_{1}\in H_{\rm{a}}^{1}\left(K_{j,m}\right), $
\begin{equation}
a\left(\mathcal{G}_{{\rm cem}}^{j,m} g, z_{1}\right)+s\left(\pi \mathcal{G}_{{\rm cem}}^{j,m} g, \pi z_{1}\right) =\int_{\partial K_{j} \cap \Gamma_{\mathrm{b}}} g\cdot z_{1}  \,\textnormal{d} s,
\end{equation}

and then summations as  $\mathcal{H}_{{\rm cem}}^{m} h=\sum_{j=1}^{N_1} \mathcal{H}_{{\rm cem}}^{j,m} h$ and $\mathcal{G}_{{\rm cem}}^{m} g=\sum_{j=1}^{N_1} \mathcal{G}_{{\rm cem}}^{j,m} g.$\\



\textbf{Part.2.The Constrained multiscale method}\\

Firstly, for local auxiliary basis function space, we find $\mathcal{H}_{{\rm cem}}^{j,m} h \in H_{\rm{a}}^{1}\left(K_{j,m}\right),$ such that for all $\theta , z_{1}, z_{2}\in$

$H_{\rm{a}}^{1}\left(K_{j,m}\right), $
\begin{equation}
\begin{aligned}
a\left(\mathcal{H}_{{\rm cem}}^{j,m} h, z_{1}\right)+s\left(\pi \mathcal{H}_{{\rm cem}}^{j,m} h, \theta \right) &=\int_{K_{j}}  \sigma \left ( h\right ):\epsilon\left ( z_{1}\right ) \,\textnormal{d} x, \\
s\left ( \mathcal{H}_{{\rm cem}}^{j,m}h,z_{2} \right )&=\int_{K_{j}}\sigma \left ( h\right ):\epsilon\left( z_{2}\right) \,\textnormal{d} x,
\end{aligned}\label{equHcem}
\end{equation}

and then summations as  $\mathcal{H}_{{\rm cem}}^{m} h=\sum_{j=1}^{N_1} \mathcal{H}_{{\rm cem}}^{j,m} h.$

Also, we find $\mathcal{G}_{{\rm cem}}^{j,m} h \in H_{\rm{a}}^{1}\left(K_{j,m}\right),$ such that for all $\theta , z_{1},z_{2} \in H_{\rm{a}}^{1}\left(K_{j,m}\right), $
\begin{equation}
\begin{aligned}
a\left(\mathcal{G}_{{\rm cem}}^{j,m} g, z_{1}\right)+s\left(\pi \mathcal{G}_{{\rm cem}}^{j,m} g, \theta \right) &=\int_{\partial K_{j} \cap \Gamma_{\mathrm{b}}} g\cdot z_{1}  \,\textnormal{d} s,\\
s\left ( H_{{\rm cem}}^{j}g,z_{2} \right )&=\int_{\partial K_{j} \cap \Gamma_{\mathrm{b}}} g\cdot z_{2}  \,\textnormal{d} s,
\end{aligned}\label{equGcem}
\end{equation}

and then summations as   $\mathcal{G}_{{\rm cem}}^{m} g=\sum_{j=1}^{N_1} \mathcal{G}_{{\rm cem}}^{j,m} g.$\\


After finishing the construction of the above steps, \myref{equinitial} is transformed into the same problem under the constrained multiscale method and the relaxed multiscale method: find $l_{{\rm cem}}^{m} \in V_{{\rm cem}}$ such that for all $v\in V_{\text{cem}},$
\begin{equation}
    a\left(l_{{\rm cem}}^{m}, v\right)=\int_{\Omega} f \cdot v \,\textnormal{d} x+\int_{\Gamma_{\mathrm{b}}} g\cdot v \,\textnormal{d} s-\int_{\Omega} \sigma\left( h\right) : \mathrm{grad}\left ( v \right ) \,\textnormal{d} x+a\left(\mathcal{H}_{{\rm cem}}^{m} h, v\right)-a\left(\mathcal{G}_{{\rm cem}}^{m} g, v\right).
\end{equation}
Finally, we give the approximation solution for \myref{equinitial} by the method above:
\begin{equation}
    u_{{\rm cem}}^{m}=l_{{\rm cem}}^{m}-\mathcal{H}_{{\rm cem}}^{m}h+\mathcal{G}_{{\rm cem}}^{m}g+h.\label{equcemu}
\end{equation}

\section{Analysis}\label{sub:3}
In a previous paper, the error analysis of the relaxed  CEM-GMsFEM was given. We will give the error analysis of the constraint CEM-GMsFEM in this paper. For clarity in the following analytical process, we give the following notation:
    \begin{align*}
        \left \| v \right \|_{a}=\left ( a\left ( v,v \right ) \right )^{\frac{1}{2}}, \quad v \in L^{2}\left(\Omega\right), \quad
        \left \| v \right \|_{a\left ( K \right )}=\left ( \int_{K} \sigma \left ( v \right ) : \varepsilon \left ( v \right )\textnormal{d} x \right )^{\frac{1}{2}}, \quad K\in\Omega ,
    \end{align*}
and
\begin{align*}
        \left \| v \right \|_{s}=\left ( s\left ( v,v \right ) \right )^{\frac{1}{2}}, \quad v \in L^{2}\left(\Omega\right),\quad \left \| v \right \|_{s\left ( K \right )}=\left ( \int_{K}\tilde{\kappa }\left | v \right |^{2}\textnormal{d} x \right )^{\frac{1}{2}}, \quad K\in\Omega.
\end{align*}
We aim to give a global range of error estimates, so the global operator is defined as $\mathcal{H}_{\rm{glo}}=\sum_{j=1}^{N_1}\mathcal{H}_{\rm{glo}}^{j},$ $\mathcal{G}_{\rm{glo}}=\sum_{j=1}^{N_1}\mathcal{G}_{\rm{glo}}^{j}$ and separately satisfy:
\begin{equation}
\begin{aligned}
a\left(\mathcal{H}_{\rm{glo}}^{j} h, z_{1}\right)+s\left(\pi \mathcal{H}_{\rm{glo}}^{j} h, \theta \right) &=\int_{K_{j}}  \sigma \left ( h\right ):\epsilon\left ( z_{1}\right ) \,\textnormal{d} x,\ \forall\ z_1\in H_{\rm{a}}^1\left(K_j\right),\\
s\left ( \mathcal{H}_{\rm{glo}}^{j}h,z_{2} \right )&=\int_{K_{j}}\sigma \left ( h\right ):\epsilon\left ( z_{2}\right ) \,\textnormal{d} x, \ \forall\ z_2\in H_{\rm{a}}^1\left(K_j\right);
\end{aligned}\label{equHwhole}
\end{equation}
\begin{equation}
\begin{aligned}
a\left(\mathcal{G}_{\rm{glo}}^{j} g, z_{1}\right)+s\left(\pi \mathcal{G}_{\rm{glo}}^{j} g, \theta \right)
&=\int_{\partial K_{j} \cap \Gamma_{\mathrm{b}}} g \cdot  z_{1}  \,\textnormal{d} s,\ \forall\ z_1\in H_{\rm{a}}^1\left(K_j\right),\\
s\left ( \mathcal{G}_{\rm{glo}}^{j}g,z_{2} \right )
&=\int_{\partial K_{j} \cap \Gamma_{\mathrm{b}}} g \cdot  z_{2}  \,\textnormal{d} s, \ \forall\ z_2\in H_{\rm{a}}^1\left(K_j\right).
\end{aligned}\label{equGwhole}
\end{equation}
And the global $l_{\rm{glo}}$ satisfies for all  $v\in V_{\rm{glo}},$
\begin{equation}
    a\left(l_{\rm{glo}}, v\right)=\int_{\Omega} f\cdot v \,\textnormal{d} x+\int_{\Gamma_{\mathrm{b}}} g\cdot v \,\textnormal{d} s-\int_{\Omega}\sigma\left( h\right) : \mathrm{grad}\left (v\right ) \,\textnormal{d} x+a\left(\mathcal{H}_{\rm{glo}} h, v\right)-a\left(\mathcal{G}_{\rm{glo}} g, v\right). \label{equ:lwhole}
\end{equation}
Thus, the global solution of \myref{equinitial} is
\begin{equation}
    u_{\rm{glo}}=l_{\rm{glo}}-\mathcal{H}_{\rm{glo}}h+\mathcal{G}_{\rm{glo}}g+h.\label{equglobalu}
\end{equation}

\begin{lemma} \label{lemma1}
By the orthogonal projection $\pi,$ we know that $\pi_{j}$ is each subdomain orthogonal projection, and give the following estimate: for all $ v\in H_{\rm{a}}^1\left(K_j\right),$
\begin{equation*}
     \left\|v-\pi_{j} v\right\|_{s\left(K_{j}\right)}^{2} \leq \frac{ \left\|v\right\|_{a\left(K_{j}\right)}^{2}}{\theta_{j}^{g_{j}+1}},
\end{equation*}
and
\begin{equation*}
    \left\|\pi_{j}v\right\|_{s\left(K_{j}\right)}^{2}=\|v\|_{s\left(K_{j}\right)}^{2}-\left\|v-\pi_{j}v\right\|_{s\left(K_{j}\right)}^{2} \leq\|v\|_{s\left(K_{j}\right)}^{2}.
\end{equation*}
\end{lemma}

To better illustrate the strength of our approach, we will introduce the following error estimation metric:
\begin{equation}
    e=l_{\rm{glo}}-\mathcal{H}_{\rm{glo}}h+\mathcal{G}_{\rm{glo}}g+h-u.
\end{equation}
Besides, we define a space
$$W_{h}=\left \{ w\in H_{\rm{a}}^{1}\left(\Omega\right) |\pi \left(w\right)=0 \right \}, $$ and for $z_{1}\in V_{\rm{glo}}$ with $a\left(w, z_{1}\right)=0.$

\begin{theorem}\label{thm:111}
Let $u$ be the real solution of \myref{equinitial}, $u_{\rm{glo}}$ be the numerical solution of \myref{equglobalu}. Then it holds that
\begin{equation}
    \left \| u_{\rm{glo}}-u \right \|_{a}=\left \| e \right \|_{a}\leq\theta ^{-\frac{1}{2}}\left \| f \right \|_{L^{2}\left ( \Omega  \right )},
\end{equation}
where $\Lambda=\underset{j}{\min }\theta_{j}^{g_{j}}.$
\end{theorem}
\begin{proof}
With \myref{equglobalu} and \myref{equcemu}, for all $z_{1}\in V_{\rm{glo}}, $ we have
\begin{equation*}
\begin{aligned}
    \int_{\Omega }&\sigma \left ( l_{\rm{glo}}-\mathcal{H}_{\rm{glo}}h+\mathcal{G}_{\rm{glo}}g+h\right ):\epsilon\left(z_{1}\right) \,\textnormal{d} x\\
    &=\int_{\Omega} f \cdot z_{1} \,\textnormal{d} x+\int_{\Gamma_{\mathrm{b}}} g \cdot z
    _{1} \,\textnormal{d} s-\int_{\Omega}\sigma\left( h\right) : \mathrm{grad}\left( z_{1} \right) \,\textnormal{d} x=a\left(u_{\rm{glo}}, z_{1}\right).
\end{aligned}
\end{equation*}
Together with above definition this implies $l_{\rm{glo}}-\mathcal{H}_{\rm{glo}}h+\mathcal{G}_{\rm{glo}}g+h-u \in W_{h}$ and therefore
\begin{equation}
    \pi \left(l_{\rm{glo}}-\mathcal{H}_{\rm{glo}}h+\mathcal{G}_{\rm{glo}}g+h-u\right)=0.\label{equpi0}
\end{equation}
Now letting $z_{h}\in W_{h}$, i.e., $\pi\left(z_{h}\right)=0$, we have
\begin{equation*}
    \begin{aligned}
        \int_{\Omega }&\sigma \left ( l_{\rm{glo}}-\mathcal{H}_{\rm{glo}}h+\mathcal{G}_{\rm{glo}}g+h-u\right ):\epsilon\left(z_{h}\right) \,\textnormal{d} x\\
        &=\int_{\Omega }\sigma \left ( l_{\rm{glo}}-\mathcal{H}_{\rm{glo}}h+\mathcal{G}_{\rm{glo}}g-\tilde{u}\right ):\epsilon\left(z_{h}\right) \,\textnormal{d} x\\
        &=a\left(l_{\rm{glo}}-\mathcal{H}_{\rm{glo}}h+\mathcal{G}_{\rm{glo}}g-\tilde{u}, z_{h}\right)\\
        &=\underset{=\mathcal{A}}{\underbrace{a\left(l_{\rm{glo}},z_{h}\right)}} -\underset{=\mathcal{B}}{\underbrace{a\left(\mathcal{H}_{\rm{glo}}h, z_{h}\right)}} +\underset{=\mathcal{C}}{\underbrace{a\left(\mathcal{G}_{\rm{glo}}g, z_{h}\right)}}
        -\underset{=\mathcal{D}}{\underbrace{a\left(\tilde{u}, z_{h}\right)}},
    \end{aligned}
\end{equation*}
where $\mathcal{A}$: $a\left(l_{\rm{glo}},z_{h}\right)=0;$

$\mathcal{B}$: $a\left(\mathcal{H}_{\rm{glo}}h, z_{h}\right)=\int_{\Omega}\sigma\left( h\right): \text{grad}\left( z_{h} \right) \,\textnormal{d} x;$

$\mathcal{C}$: $a\left(\mathcal{G}_{\rm{glo}}g, z_{h}\right)=\int_{\partial \Omega \cap \Gamma_{\mathrm{b}}} g \cdot z_{h}  \,\textnormal{d} s;$

$\mathcal{D}$: $a\left(\tilde{u}, z_{h}\right)=\int_{\Omega} f \cdot  z_{h}\textnormal{d} x+\int_{\Gamma_{\mathrm{b}}} g \cdot z_{h}\textnormal{d}\sigma-\int_{\Omega}\sigma\left( h\right) : \mathrm{grad}\left ( z_{h} \right ) \,\textnormal{d} x.$\\
Hence,
\begin{equation*}
    \begin{aligned}
    \int_{\Omega }&\sigma \left ( l_{\rm{glo}}-\mathcal{H}_{\rm{glo}}h+\mathcal{G}_{\rm{glo}}g+h-u\right ):\epsilon\left(z_{h}\right) \,\textnormal{d} x\\
    &=-\int_{\Omega}\sigma\left( h\right): \text{grad}\left ( z_{h} \right ) \,\textnormal{d} x+\int_{\partial \Omega \cap \Gamma_{\mathrm{b}}} g \cdot z_{h}  \,\textnormal{d} s\\
    &\quad -\int_{\Omega} f \cdot z_{h} \,\textnormal{d} x-\int_{\Gamma_{\mathrm{b}}} g \cdot z_{h} \,\textnormal{d} s+\int_{\Omega}\sigma\left( h\right) : \mathrm{grad}\left ( z_{h} \right ) \,\textnormal{d} x\\
    &= -\int_{\Omega} f \cdot z_{h} \,\textnormal{d} x.
\end{aligned}
\end{equation*}

Using \myref{equpi0}, we can choose $z_{h}=e=l_{\rm{glo}}-\mathcal{H}_{\rm{glo}}h+\mathcal{G}_{\rm{glo}}g+h-u$ to obtain
\begin{equation*}
    \begin{aligned}
        \left \|  u_{\rm{glo}}-u \right \|_{a}^{2}
        &=\int_{\Omega} f \cdot e  \\
       &=\int_{\Omega}f \cdot \left ( e-\pi e \right )\\
       &\leq \left \| f \right \|_{L^{2}\left ( \Omega  \right )}\left \|  e -\pi e \right \|_{L^{2}\left ( \Omega  \right )}\\
       &\leq \left ( \Lambda \right )^{-\frac{1}{2}}\left \| f \right \|_{L^{2}\left ( \Omega  \right )}\left \|  e  \right \|_{a}.
    \end{aligned}
\end{equation*}
\end{proof}

{\rm In order to localize the multiscale space, we use the following definition.}

\begin{definition}
 For $K \in \mathcal{T}_{H}$, we write $U\left(K\right)$ as the extension of $K$, if $K \subset U\left(K\right) \subset \Omega$, and $U\left(K\right)$ satisfies
$$U\left(K\right)=\operatorname{int} \bigcup_{\mathcal{T} \in \mathcal{T}_{h}^{*}} \overline{\mathcal{T}}, \quad \text { where } \mathcal{T}_{h}^{*} \subset \mathcal{T}_{h}.$$
And we give the definition of the closure set and open kernel set,
$$W_{h,1}=\left \{  w\in W_{h}: w=0\  {\rm in}\  U\left ( K \right ) \cap \left( \Omega  \backslash K\right)\right  \},$$
$$W_{h,2}=\left \{  w\in W_{h}: w=0\  {\rm in}\   \Omega  \backslash  U\left ( K \right )\right \}.$$
\end{definition}
\begin{definition}
 For $K_i \in \mathcal{T}_{H},$  $m,n\in N ,  n<m, $ a cutoff function $\beta_{i}^{m,n}\in V_{H}$ which is a Lagrange basis function space of $ \mathcal{T}^{H} $  and satisfies
 \begin{equation}
\begin{aligned}
\beta_{i}^{m,n}\left(x\right) & \equiv 0, \quad  {\rm in }\  \Omega \backslash K_{i,m}, & \\
\beta_{i}^{m,n}\left(x\right) & \equiv 1, \quad  {\rm in }\  K_{i,n},& \\
0 \leq \beta_{i}^{m,n}\left(x\right) & \leq 1, \quad  {\rm in}\  K_{i,m} \backslash K_{i,n}.&
\end{aligned}\label{equacutoff}
 \end{equation}
\end{definition}

{\rm We introduce completeness property.}
\begin{lemma}\label{proplemma1}
For $\forall$ $z_{2} \in V_{\rm{aux}},$ there exists a constant $A$ and $z_{1}\in H_{\rm{a}}^{1}\left(\Omega\right) $ such that
\begin{equation*}
    \pi\left(z_{1}\right)=z_{2},\quad \left \| z_{1} \right \|_{a}\leq A \left \|z_{2} \right \|_{s},\quad \text{supp}\left(z_{1}\right)\subset \text{supp}\left(z_{2}\right).
\end{equation*}
\begin{proof}
We could consider the space $V_{\rm{aux}}\left(K_{i}\right)$ and the following problem: find $\xi \in H_{\rm{a}}^{1}\left(K_{i}\right)$ and $z_{1}\in V_{\rm{aux}}\left(K_{i}\right),$ such that
\begin{equation}
    \begin{aligned}
        \int_{K_{i}}\sigma \left ( z_{1}\right ):\varepsilon \left ( v \right ) \,\textnormal{d} x+\int_{K_{i}}\widetilde{k}\xi \cdot v \textnormal{d} x &=0,\quad \quad \quad  \quad \quad  \forall v \in H_{\rm{a}}^{1}\left ( K_{i} \right ),\\
        \int_{K_{i}}\widetilde{k}z_{1}\cdot q \,\textnormal{d} x&=\int_{K_{i,m}}\widetilde{k}z_{2}\cdot q \,\textnormal{d} x,\quad \forall q \in V_{\rm{aux}}\left(K_{i}\right).\label{equlemma1}
    \end{aligned}
\end{equation}
We first proof the completeness of \myref{equlemma1}. For all coarse basis functions $\hat{z_{2}}\in H_{\rm{a}}^{1}\left(\Omega\right),$ we search for $\hat{z_{1}}\in V_{\rm{aux}}$ with
\begin{equation*}
     \pi\left(\hat{z_{1}}\right)=\hat{z_{2}},\quad \left \|  \hat{z_{1}}\left(x\right)  \right \|_{a}\leq A_{3}\left \| \hat{z_{2}}\left(x\right) \right \|_{s},\quad \text{supp}\left(\hat{z_{1}}\right)\subset \text{supp}\left(\hat{z_{2}}\right).
\end{equation*}
It suffices to show that the following function we defined has desired properties,
\begin{equation*}
    z_{1}:=z_{2}+\sum_{x\in K_{i}} \left(z_{2}\left(x\right)-\pi\left(z_{2}\right) \left(x\right)  \right)\hat{z_{1}}.
\end{equation*}
It holds that
\begin{equation*}
    \int_{K_{i}}\widetilde{k}z_{1}\cdot z_{2} \,\textnormal{d} x=\int_{K_{i}}\widetilde{k}\left [ z_{2}+\sum_{x\in K_{i}} \left(z_{2}\left(x\right)-\pi\left(z_{2}\right) \left(x\right)  \right)\hat{z_{1}} \right ]z_{2} \,\textnormal{d} x \geq A_{4}\left \|  z_{2} \right \|_{a\left(K_{i}\right)}^{2},
\end{equation*}
where $$A_{4}=\underset{\xi\neq 0,\xi\in V_{\rm{aux}}}{\underset{K\in\mathcal{T}_{H}}{\text{inf}}}\frac{\int_{T}\tilde{k}z_{1}\cdot \xi \textnormal{d} x}{\left \| \xi \right \|_{s}}.$$
Actually, we find a function $ z_{1}\in H_{\rm{a}}^{1}\left(K_{i} \right ) $ such that
\begin{equation*}
    \left\{\begin{matrix}
\int_{K_{i}}\widetilde{k}z_{1}\cdot z_{2} \,\textnormal{d} x\geq A_{1}\left \| z_{2} \right \|_{s\left ( K_{i}\right )}^{2}\\
\left \| z_{1} \right \|_{a\left ( K_{i}\right )}^{2}\leq A_{2}\left \| z_{2} \right \|_{s\left ( K_{i}\right )}^{2}
\end{matrix}\right.,
\end{equation*}
for constant $A_{1},A_{2}.$
We finish this proof.
\end{proof}
\end{lemma}

\begin{lemma}\label{proplemma2}
For a given $ \hat{w} \in W_{h}$ and a given cutoff function $\beta_{i}^{m,n} $ defined in \myref{equacutoff} and $m\geq n \geq 0$, there exists some ${\hat{w}}^{*} \in{W}_{h,2}\left(\Omega \backslash K_{i,m-n-1}\right) \subset W_{h}$ such that
\begin{equation}
\begin{aligned}
    & \quad \left \| \beta_{i}^{m,n}\hat{w} -\hat{w}^{*} \right \|_{a}+\left \| \pi\left(\beta_{i}^{m,n}\hat{w} -\hat{w}^{*} \right) \right \|_{s}
     \leq \frac{\left \|  \hat{w} \right \|_{a\left ( K_{i,m+2} \backslash K_{i,m-n+2}  \right )}+\left \|  \pi\hat{w} \right \|_{s\left ( K_{i,m+2}\backslash K_{i,m-n+2}  \right )}}{n}.
\end{aligned}
\end{equation}
\begin{proof}
Consider the following problem: find $\xi \in H_{\rm{a}}^{1}\left(K_{i,m}\right)$ and $\beta_{i}^{m,n}\hat{w}\in V_{\rm{aux}}\left(K_i\right),$ such that
\begin{equation}
    \begin{aligned}
        \int_{K_{i,m}}\sigma \left ( \beta_{i}^{m,n}\hat{w}\right ):\varepsilon \left ( v \right ) \,\textnormal{d} x+\int_{K_{i,m}}\widetilde{k}\xi \cdot v \,\textnormal{d} x&=0,\quad \quad \quad  \quad \quad  \forall v \in H_{\rm{a}}^{1}\left ( K_{i,m} \right ),\\
        \int_{K_{i,m}}\widetilde{k}\beta_{i}^{m,n}\hat{w} \cdot q \,\textnormal{d} x&=\int_{K_{i,m}}\widetilde{k}z_{2} \cdot q \,\textnormal{d} x,\quad \forall q \in V_{\rm{aux}}\left ( K_{i,m} \right ).
    \end{aligned}
\end{equation}
Introducing the operator ${\pi}': H_{\rm{a}}^1\left(\Omega\right) \rightarrow L^{2}$  and by Lemma \ref{proplemma1}, there exists $v\in H_{\rm{a}}^{1}\left(\Omega\right) $  such that
\begin{equation}
\begin{aligned}
     &v= {\pi}'\left(\beta_{n}\hat{w}\right),\left \| v \right \|_{a\left ( \Omega \right )}\leq \left \|  {\pi}'\left(\beta_{n}\hat{w}\right)   \right \|_{a\left ( \Omega \right )}, \\
    & {\rm supp}\left ( v \right )\subset {\rm supp}\left ( \beta _{n} \right )\subset \Omega \verb|\| K_{i,m-n-1}.
\end{aligned}
\end{equation}
Then we let
\begin{equation*}
    \begin{aligned}
     \beta_{n} & :=\beta_{i}^{m,n},\\
     \hat{w}^{*}&:={\pi}'\left(\beta_{{m}'}\hat{w}\right)-v\in W_{h,2}\left ( \Omega \verb|\|K_{i,m-{m}'-1} \right ),
    \end{aligned}
\end{equation*}
and $$b_{K}^{n}=\frac{\int_{K_{i,m}}\beta _{n} \,\textnormal{d} x}{\left | K_{i,m} \right |}.$$
Knowing that $\pi{\pi}'\left(\hat{w}\right)=\pi\left(\hat{w}\right)=0,$  for any $T\in\mathcal{T}_{H}, $ we have
\begin{equation}
    \begin{aligned}
      &\quad \left \| {\pi}'\left ( \beta _{n}\hat{w} \right ) \right \|_{a\left ( T \right )}+\left \| {\pi}'\pi\left ( \beta _{n}\hat{w} \right ) \right \|_{s\left ( T\right )}\\
      &=\left \| {\pi}'\left ( \left ( \beta _{n}-b_{K}^{n}\right ) \hat{w} \right ) \right \|_{a\left ( T \right )}+\left \| {\pi}'\pi\left ( \left ( \beta _{n}-b_{K}^{n}\right ) \hat{w} \right ) \right \|_{s\left ( T\right )}\\
      &\leq \left \|  \left ( \beta _{n}-b_{K}^{n}\right ) \hat{w}  \right \|_{a\left ( T \right )}+\left \| \pi \left ( \beta _{n}-b_{K}^{n}\right ) \hat{w}  \right \|_{s\left ( T \right )}.
    \end{aligned}
\end{equation}
Next we estimate the term $\left \| {\pi}'\left ( \beta _{n}\hat{w} \right ) \right \|_{a\left ( \Omega \right )}+\left \| {\pi}'\pi\left ( \beta _{n}\hat{w} \right ) \right \|_{s\left ( \Omega \right )}.$
\begin{equation}
\begin{aligned}
     &\quad \left \| {\pi}'\left ( \beta _{n}\hat{w} \right ) \right \|_{a\left ( \Omega \right )}+\left \| {\pi}'\pi\left ( \beta _{n}\hat{w} \right ) \right \|_{s\left ( \Omega \right )}\\
     & \leq  \underset{T\in K_{i,m+1}\backslash K_{i,m-n+1}}{\sum_{T \in \mathcal{T}_{H}:}}\left\|\left(\beta _{n}-b_{K}^{n}\right ) \hat{w}\right\|_{a\left(T\right)}^{2}+\underset{T\in K_{i,m+1}\backslash K_{i,m-n+1}}{\sum_{T\in\mathcal{T}_{H}:}}\left\|\pi\left(\beta _{n}-b_{K}^{n}\right)\hat{w}\right\|_{s\left(T\right)}^{2}\\
     &\leq  \sum_{T\in\mathcal{T}_{H}}\int_{K_{i,m+1} \backslash K_{i,m-n+1}}\sigma \left ( \left ( \beta _{n}-b_{K}^{n}\right ) \hat{w}  \right ):\varepsilon \left ( \left ( \beta _{n}-b_{K}^{n}\right ) \hat{w} \right ) \,\textnormal{d} x\\
     &\quad + \sum_{T\in\mathcal{T}_{H}}\int_{K_{i,m+1} \backslash K_{i,m-n+1}}\tilde{k}\left ( \left ( \beta _{n}-b_{K}^{n}\right ) \hat{w}  \right )\cdot\left ( \left ( \beta _{n}-b_{K}^{n}\right ) \hat{w}  \right ) \,\textnormal{d} x\\
     &\leq \underset{T\subset K_{i,m}\backslash K_{i,m-n}}{\sum_{T \in \mathcal{T}_{H}:}}\left \| \beta _{n} \left ( \hat{w}-\pi\hat{w} \right ) \right \|_{a\left ( T\right )}^{2} +\underset{T\subset K_{i,m+1}\backslash K_{i,m-n-1}}{\sum_{T \in \mathcal{T}_{H}:}}\left \| \pi\left ( \beta _{n}-b_{K}^{n}\right ) \hat{w} \right \|_{a\left ( T \right )}^{2}\\
     &\quad +\underset{T\subset K_{i,m}\backslash K_{i,m-n}}{\sum_{T \in \mathcal{T}_{H}:}}\left \| \pi \beta _{n} \left ( \hat{w}-\pi\hat{w} \right ) \right \|_{s\left ( T\right )}^{2}+\underset{T\subset K_{i,m+1}\backslash K_{i,m-n-1}}{\sum_{T \in \mathcal{T}_{H}:}}\left \| \pi\left(\pi\left ( \beta _{n}-b_{K}^{n}\right ) \hat{w}\right) \right \|_{s\left ( T \right )}^{2}\\
    &\leq \left ( \frac{1}{N_{y}} \right )^{2}\left \| \beta _{n} \right \|_{a\left ( \Omega  \right )}^{2}\left \| \hat{w} \right \|_{a\left ( K_{i,m+1}\backslash K_{i,m-n-1} \right )}^{2} +\underset{T\subset K_{i,m+1}\backslash K_{i,m-n-1}}{\sum_{T \in \mathcal{T}_{H}:}}\left \| \pi\left ( \beta _{n}-b_{K}^{n}\right ) \hat{w}   \right  \|_{a\left ( T \right )}^{2}\\
     &\quad +\left ( \frac{1}{N_{y}} \right )^{2}\left \| \beta _{n}  \right \|_{a\left ( \Omega  \right )}^{2}\left \|\pi \hat{w} \right \|_{s\left ( K_{i,m+1}\backslash K_{i,m-n-1} \right )}^{2}+\underset{T\subset K_{i,m+1}\backslash K_{i,m-n-1}}{\sum_{T \in \mathcal{T}_{H}:}}\left \| \pi\left(\pi \left ( \beta _{n}-b_{K}^{n}\right )  \hat{w}\right)\right \|_{s\left ( T \right )}^{2}\\
      &\leq \left ( \frac{1}{N_{y}} \right )^{2}\left \| \beta _{n}\right \|_{a\left ( \Omega  \right )}^{2}\left \| \hat{w} \right \|_{a\left ( K_{i,m+2}\backslash K_{i,m-n-2} \right )}^{2}+\left ( \frac{1}{N_{y}} \right )^{2}\left \| \beta _{n} \right \|_{a\left ( \Omega  \right )}^{2}\left \|\pi \hat{w} \right \|_{s\left ( K_{i,m+2}\backslash K_{i,m-n-2} \right )}^{2},
\end{aligned}
\end{equation}
where we use the inequality
\begin{equation*}
\begin{aligned}
     \left \| \beta _{n}-b_{K}^{n} \right \|_{a\left ( T \right )}^{2}&\leq \frac{1}{\left(N_{y}\right)^2}\left \| \beta _{n}  \right \|_{a\left ( \Omega  \right )}^{2},\\
     \left \| \beta _{n}-b_{K}^{n} \right \|_{s\left ( T \right )}^{2}&\leq \frac{1}{\left(N_{y}\right)^2}\left \| \beta _{n} \right \|_{s\left ( \Omega  \right )}^{2},
\end{aligned}
\end{equation*}
and assume $N_{y}\leq N_{x},n_{y}\leq n_{x}.$
Taking ${\pi}'\left(\hat{w}\right)=\hat{w}$ and $\pi\left(\hat{w}\right)=0,$ we estimate $\left\| \beta _{n} \left(\hat{w}-{\pi}'\hat{w}\right)\right\|_{a\left(\Omega \right)}^{2}$ and $\left\| \beta_{n}\hat{w}-{\pi}'\hat{w}\right\|_{a\left(\Omega \right)}^{2}$ respectively.
\begin{equation*}
    \begin{aligned}
         &\quad \left\|\beta_{n}\left(\hat{w}-{\pi}'\hat{w}\right)\right\|_{a\left(\Omega \right)}^{2}\\
         &=\sum_{T\in \mathcal{T}_{H}}\left \|  \left ( \beta _{n}-b_{T}^{n}  \right ) \hat{w}-{\pi}'\left ( \beta _{n}-b_{T}^{n}  \right ) \hat{w} \right \|_{a\left ( T \right )}^{2}\\
         &\leq \frac{1}{\left ( N_{x}\cdot n_{x} \right )^{2}}\sum_{T\in \mathcal{T}_{H}}\left \|  \left ( \beta _{n}-b_{T}^{n}  \right )\hat{w}\right \|_{a\left ( T \right )}^{2}+\left \| \left ( \left ( \beta _{n}-b_{T}^{n}  \right ) \hat{w} \right )\left ( {\pi}'\left ( \beta _{n}-b_{T}^{n}  \right ) \hat{w}  \right ) \right \|_{a\left ( T \right )}\\
         &\quad +\sum_{X\in T}\left \| \left( \left ( \beta _{n}-b_{T}^{n}  \right )\hat{w}\right)\right \|_{a\left ( X \right )}^{2}\\
         & \leq \frac{1}{\left ( N_{x}\cdot n_{x} \right )^{2}}\sum_{T\in \mathcal{T}_{H}}\left\| \beta _{n} \left(\hat{w}-\pi\hat{w}\right)\right\|_{a\left(T\right)}^{2}+\left \|\beta _{n} \varepsilon \left ( \hat{w}  \right )\right \|_{a\left ( T \right )}^{2}+\sum_{X\in T}\left \| \left ( \beta _{n}-b_{T}^{n}  \right )\hat{w}\right \|_{a\left ( X \right )}^{2}\\
         &\leq \frac{1}{\left ( N_{x}\cdot n_{x} \right )^{2}}\underset{T\subset K_{i,m+1} \backslash K_{i,m-n-1}}{\sum_{T\in \mathcal{T}_{H}}}\left\|\beta _{n}\right\|_{a\left(T\right)}^{2}\left\|\hat{w} \right\|_{a\left(T\right)}^{2}+\frac{1}{\left ( N_{y} \right )^{2}}\left \| \beta _{n}  \right \|_{a\left ( X \right )}^{2}\sum_{X\in T}\left ( n_{x} \right )^{2}\left \|\hat{w}\right \|_{a\left ( X \right )}^{2}\\
         &\leq \left ( \frac{1}{N_{y}} +\frac{1}{n_{y}} \right )\left \| \beta _{n}\right \|_{a\left ( \Omega \right )}^{2}\cdot \left \|\hat{w} \right \|_{a\left ( K_{i,m+1} \backslash K_{i,m-n-1} \right )}^{2}.
    \end{aligned}
\end{equation*}
Combining above inequalities, we note that
\begin{equation*}
    \begin{aligned}
     &\quad \left \| \beta _{n}\hat{w} -\hat{w}^{*} \right \|_{a\left ( \Omega \right )}^{2}\\
     &\leq \left \|\beta _{n}\hat{w} -{\pi}'\left (\beta _{n}\hat{w}  \right ) \right \|_{a\left ( \Omega \right )}^{2}+\left \| {\pi}'\left ( \beta _{n}\hat{w} \right ) \right \|_{a\left ( \Omega \right )}^{2}\\
     &\leq \left ( \frac{1}{\left ( N_{y}\cdot n_{y} \right )^{2}}\left \| \beta _{n}  \right \|_{a\left ( \Omega\right )}^{2}+\frac{1}{\left ( N_{y} \right )^{2}}\left \|\beta _{n}  \right \|_{a\left ( \Omega  \right )}^{2} \right )\cdot \left \|  \hat{w} \right \|_{a\left ( K_{i,m+2} \backslash K_{i,m-n+2}  \right )}^{2}\\
     &\leq \frac{1}{n^{2}}\left \| \hat{w} \right \|_{a\left ( K_{i,m+2}\backslash K_{i,m-n+2}  \right )}^{2}.
    \end{aligned}
\end{equation*}
Similarly, we have
\begin{equation*}
     \left\|\beta_{n}\left(\hat{w}-{\pi}'\hat{w}\right)\right\|_{a\left(\Omega \right)}^{2}\leq \frac{1}{n^{2}}\left \|\hat{w}\right \|_{a\left ( K_{i,m+2}\backslash K_{i,m-n+2}  \right )}^{2}.
\end{equation*}
This ends the proof.
\end{proof}
\end{lemma}
{\rm Remark: In the rest of the paper, we use the operator ${\pi}': H_{\rm{a}}^1\left(\Omega\right) \rightarrow L^{2}\left(\Omega\right)$.}

\begin{lemma}\label{proplemma3}
Let $ w^{K}\in W_{h}$ be the solution of $\int_{\Omega}\sigma \left ( w^{K} \right ):\varepsilon \left ( v \right ) \,\textnormal{d} x=\mathcal{B}_{K}\left ( v \right ),$ for all $v\in W_{h},$
where $\mathcal{B}_{K}$ is such that $\mathcal{B}_{K}\left ( v \right ) =0$ for all $v\in W_{h,2}\left ( \Omega \backslash K \right ).$ Then there exists a constant $0\leq \delta \leq 1 $ such that
\begin{equation}
    \left \|  w^{K} \right \|_{a\left ( \Omega\backslash K_{i,m} \right )}+\left \|  \pi w^{K} \right \|_{s\left ( \Omega\backslash K_{i,m} \right )}
    \leq \delta ^{m}\left(\left \|  w^{K}  \right \|_{a\left ( \Omega  \right )}+\left \| \pi w^{K}  \right \|_{s\left ( \Omega  \right )}\right).
\end{equation}

\begin{proof}
We denote $\beta_{n}:= \beta _{K}^{m,n}.$ Using the extension of Lemma \ref{proplemma2} gives us that
\begin{equation*}
    \left \| \beta _{n}w^{K}- \hat{w}^{K} \right \|_{a\left ( \Omega \right )}\leq \frac{\left \|  w^{K} \right \|_{a\left ( K_{i,m}\backslash K_{i,m-n}\right )}}{n}, \quad \forall \hat{w}^{K}\in W_{h,2}\left(\Omega\backslash K_{m-n-1}\right).
\end{equation*}
Note that $\hat{w}^{K}\in W_{h,2}\left(\Omega\backslash K\right),$ we also have
\begin{equation}
    \int_{\Omega\backslash K_{i,m-n}}\sigma \left ( w^{K} \right ): \varepsilon \left ( \hat{w} ^{K}\right ) \,\textnormal{d} x=\int_{\Omega}\sigma \left ( w^{K} \right ): \varepsilon \left ( \hat{w} ^{K}\right ) \,\textnormal{d} x=\mathcal{B}_{K}\left ( \hat{w} ^{K} \right )=0.
\end{equation}
Then we obtain
\begin{equation*}
    \begin{aligned}
     &\quad \int_{\Omega \backslash K_{i,m}}\sigma \left ( w^{K} \right ):\varepsilon \left ( w^{K} \right ) \,\textnormal{d} x\\
     &\leq \int_{\Omega \backslash K_{i,m-n}}\beta _{n}\sigma \left ( w^{K} \right ):\varepsilon \left ( w^{K} \right ) \,\textnormal{d} x\\
    &\leq \int_{\Omega \backslash K_{i,m-n}}\sigma \left ( w^{K} \right ):\left ( \varepsilon \left(\beta _{n}\left ( w^{K} \right )\right)-  w^{K} \varepsilon \left ( \beta _{n} \right )\right ) \,\textnormal{d} x\\
    &=\int_{\Omega \backslash K_{i,m}}\sigma \left ( w^{K} \right ):\left ( \varepsilon \left ( \beta _{n} w^{K}-\hat{w}^{K}\right )-\left ( w^{K} -\pi\left ( w^{K} \right )\right )\varepsilon \left ( \beta _{n} \right ) \right ) \,\textnormal{d} x\\
    &\leq \frac{1}{n}\left ( \left \| w^{K}  \right \|_{a\left ( \Omega \backslash K_{i,m-1} \right )}^{2} +\frac{1}{N_{y}}\left \| w^{K}  \right \|_{a\left ( \Omega \backslash K_{i,m-n} \right )}\left \| w^{K}-\pi\left ( w^{K} \right ) \right \|_{a\left ( \Omega \backslash K_{i,m-n} \right )}^{2}\right )\\
    &\leq \frac{\left \| w^{K}  \right \|_{a\left ( \Omega \backslash K_{i,m-n-1} \right )}^{2}}{n}.
    \end{aligned}
\end{equation*}
The following inequality holds,
\begin{equation}
   \left \| w^{K} \right \|_{a\left ( \Omega \backslash K_{i,m} \right )} \leq d \frac{1}{n}\left \| w^{K}  \right \|_{a\left ( \Omega \backslash K_{i,m-n-1} \right )}^{2},
\end{equation}
where $d$ is a constant. By choosing $n:=\left [ \frac{d\cdot 2n}{m-n} \right ],$ we have
\begin{equation*}
    \left \| w^{K} \right \|_{a\left ( \Omega \backslash K_{i,m} \right )} \leq \delta \left \| w^{K}  \right \|_{a\left ( \Omega\right )}^{2},
\end{equation*}
where $\delta =2^{\frac{m\left ( n+1 \right )}{n\left ( m+3 \right )}}.$ We could get the estimate of $ \left \| w^{K}  \right \|_{a\left ( \Omega \backslash K_{i,m} \right )} . $ This ends the proof.
\end{proof}
\end{lemma}

{\rm The following lemma shows that the multiscale basis functions for velocity has a similar decay property. We emphasize that in our proof using Lemma \ref{proplemma1}-Lemma \ref{proplemma3} below, one to construct test functions that are based on constraint CEM-GMsFEM, which  differs from the proof of before case.}

\begin{lemma}\label{lemma5}
Let $w^{K}\in W_{h}$ be the solution of
\begin{equation}
    \int_{\Omega }\sigma \left ( w^{K} \right ):\varepsilon \left ( v \right ) \,\textnormal{d} x=\mathcal{B}_{K}\left ( v \right ), \quad  \forall  v \in H_{\rm{a}}^{1}\left(\Omega\right),
\end{equation}
where $\mathcal{B}_{K}\in W_{h,1}$ is such that $\mathcal{B}_{K}\left(v\right)=0$ for $\forall v\in W_{h,2}.$ Moreover, $w_{i,m}^{K}\in W_{h,1}\left(U\left(K_{i,m}\right)\right)$ is denoted as the solution of
\begin{equation}
    \int_{K_{i,m} }\sigma \left ( w_{i,m}^{K} \right ):\varepsilon \left ( v \right ) \,\textnormal{d} x=\mathcal{B}_{K}\left ( v \right ), \quad  \forall  v \in W_{h,1}\left(U\left(K_{i,m}\right)\right).
\end{equation}
Then there exists a positive constant $ 0\leq \delta \leq 1 $ such that
\begin{equation}
    \left \| \sum_{K\in \mathcal{T}_{H}}\left ( w^{K}-w_{i,m}^{K} \right ) \right \|_{a}^{2}+\left \| \sum_{K\in \mathcal{T}_{H}}\pi\left ( w^{K}-w_{i,m}^{K} \right ) \right \|_{s}^{2}
    \leq m^{d}\delta^{m}\left(\sum_{K\in \mathcal{T}_{H}}\left \| w^{K}\right \|_{a}^{2}+\sum_{K\in \mathcal{T}_{H}}\left \| \pi w^{K}\right \|_{s}^{2}\right).
\end{equation}
\end{lemma}

\begin{proof}
Take $y_{m}:=w^{K}-w_{i,m}^{K},$ and $y:=\sum_{K\in \mathcal{T}_{H}} \left ( w^{K}-w_{i,m}^{K} \right ),$
we note that
\begin{equation*}
    \sigma\left(y\right)=\sum_{K\in \mathcal{T}_{H}} \underset{:=\sigma\left ( y_{1} \right )}{\underbrace{\sigma \left(\left (w^{K}-w_{i,m}^{K} \right ) \left ( 1-\beta_{K}^{m,1} \right )y \right )}}+\sum_{K\in \mathcal{T}_{H}}\underset{:=\sigma\left(y_{2} \right )}{\underbrace{ \sigma \left ( \left ( w^{K}-w_{i,m}^{K} \right )\beta_{K}^{m,1}y\right )}}.
\end{equation*}
In the following we solve for each of the two parts:
\begin{equation*}
\begin{aligned}
    \left \| y_1 \right \|_{a}&\leq \left \| y_{m}\right \|_{a}\cdot\left \| y \left ( 1-\beta _{K}^{m,1} \right ) \right \|_{a\left(K_{i,m+1} \right)}\\
    &\leq \left \| y_{m}\right \|_{a}\cdot\left ( \left \| y  \right \|_{a\left ( K_{i,m+1} \right )}+\left \| y\left ( 1-\beta _{K}^{m,1} \right ) \right \|_{a\left ( K_{i,m+1}\backslash K_{i,m}\right )} \right )\\
    &\leq \left \| y_{m}\right \|_{a}\left \| y \right \|_{a\left ( K_{i,m+2} \right )}.
\end{aligned}
\end{equation*}
Then for $\tilde{y}\in W_{h,2},$ we have 
$$ \int _{\Omega }\sigma \left ( y_{m} \right ):\varepsilon \left ( \tilde{y} \right ) \,\textnormal{d} x=\mathcal{B}_{K}\left ( \tilde{y} \right )=0, \text{and}\ \left \| y\beta _{K}^{m,1} -\tilde{y}\right \|_{a\left ( K_{i,m+2} \right )}\leq \left \| y\right \|_{a\left ( K_{i,m+2} \right )}.$$
Hence,
$$\left \| y_{2} \right \|_{a}\leq \left \| y_{m}\right \|_{a}\left \| y \right \|_{a\left ( K_{i,m+2} \right )}$$
Then we note that
\begin{equation*}
    \begin{aligned}
        \left \| y \right \|_{a}^{2}&\leq \sum_{K\in \mathcal{T}_{H}}\left \| y_{m} \right \|_{a}\left \| y \right \|_{a\left ( K_{i,m+2} \right )}\\
        &\leq m^{d}\left ( \sum_{K\in \mathcal{T}_{H}}\left \| y_{m} \right \|_{a} \right )\left \| y \right \|_{a}.
    \end{aligned}
\end{equation*}
Next we estimate $\left \| y_{m} \right \|_{a}.$
Note that
$$\left \| y_{m}\right \|_{a}^{2}\leq \underset{\hat{w}\in W_{h,1}\left ( K_{i,m} \right )}{\text{inf}}\left \| w^{K}-\hat{w} \right \|_{a}^{2},$$
there exists the existence of ${v}'\in H_{\rm{a}}^{1},$ such that
\begin{equation*}
    \begin{aligned}
      \left\{\begin{matrix}
&\pi{v}'=\pi{\pi}'\left ( \left ( 1-\beta _{K}^{m,1} \right )w^{K} \right )\\
&\left \| {v}' \right \|_{a}\leq \left \| \pi{\pi}'\left ( \left ( 1-\beta _{K}^{m,1} \right )w^{K} \right ) \right \|_{a}\\
&\text{supp}\left ( {v}' \right )\subset \text{supp}\left ( \left ( 1-\beta _{K}^{m,1} \right )w^{K} \right )\subset K_{i,m}
\end{matrix}\right.
    \end{aligned}.
\end{equation*}
Next we define $\hat{w}:={\pi}'\left ( \left ( 1-\beta _{K}^{m,1} \right )w^{K} \right )-{v}'\in W_{h,1}\left(K_{i,m} \right ),$ such that
\begin{equation*}
    \begin{aligned}
      &\quad \left \| \pi{\pi}' \left ( \left ( 1-\beta _{K}^{m,1} \right )w^{K} \right )\right \|_{a\left ( K_{i,m} \right )}\\
      &=\left \| \pi{\pi}' \left ( \left ( 1-\beta _{K}^{m,1} \right )w^{K} \right )\right \|_{a\left ( K_{i,m}\backslash K_{i,m-2} \right )}+\left \| \pi{\pi}' \left ( \left ( 1-\beta _{K}^{m,1} \right )w^{K} \right )\right \|_{a\left (K_{i,m-2} \right )}\\
     &=\left \| \pi{\pi}' \left ( \left ( 1-\beta _{K}^{m,1} \right )w^{K} \right )\right \|_{a\left ( K_{i,m}\backslash K_{i,m-2} \right )}+\left \| \pi \left(w^{K} \right )\right\|_{a\left (K_{i,m-2} \right )}\\
     &=\left \| \pi{\pi}' \left( \left( 1-\beta _{K}^{m,1} \right)w^{K} \right)\right \|_{a\left(K_{i,m} \backslash K_{i,m-2} \right)},
    \end{aligned}
\end{equation*}
and $$\left \| \left ( 1-\beta _{K}^{m,1} \right ) w^{K}-{\pi}' \left ( \left ( 1-\beta _{K}^{m,1} \right ) \right )w^{K}\right \|_{a\left ( K_{i,m} \right )}^{2}\leq \left \| w^{K} \right \|_{a\left ( K_{i,m+1}\backslash K_{i,m-2} \right )}.$$
We obtain that
\begin{equation*}
    \begin{aligned}
      \left \| y_{m} \right \|_{a}^{2}&\leq \left \| ( \beta _{K}^{m,1}w^{K}+ \left( 1-\beta _{K}^{m,1} \right )w^{K}- \hat{w}^{K}\right \|_{a}^{2}\\
     &\leq \left \| w^{K} \right \|_{a\left ( \Omega \backslash K_{i,m} \right )}+\left \| w^{K} \right \|_{a\left ( K_{i,m+1}\backslash K_{i,m-2} \right )}+\left \| \hat{w}^{K} \right \|_{a\left ( K_{i,m} \right )}\\
     &\leq \left \| w^{K} \right \|_{a\left ( \Omega \backslash K_{i,m} \right )}+\left \| w^{K} \right \|_{a\left ( K_{i,m+1}\backslash K_{i,m-2} \right )}+\left \| \pi {\pi}'\left ( \left (  1-\beta _{K}^{m,1} \right ) w^{K}\right ) \right \|_{a\left ( K_{i,m} \right )}\\
     &\leq \left \| w^{K} \right \|_{a\left ( \Omega \backslash K_{i,m-2} \right )}+\left \| \pi {\pi}'\left ( \left (  1-\beta _{K}^{m,1} \right ) w^{K}\right ) \right \|_{a\left ( K_{i,m}\backslash K_{i,m-2} \right )}\\
     &\leq \left \| w^{K} \right \|_{a\left ( \Omega \backslash K_{i,m-3} \right )}\\
     &\leq \delta ^{2\left ( m-3 \right )}\left \| w^{K} \right \|_{a}^{2}.
    \end{aligned}
\end{equation*}
Similarly, we have
\begin{equation}
    \left \| \sum_{K\in \mathcal{T}_{H}}\pi\left ( w^{K}-w_{i,m}^{K} \right ) \right \|_{s}^{2}
    \leq m^{d}\delta^{m}\sum_{K\in \mathcal{T}_{H}}\left \| \pi w^{K}\right \|_{s}^{2}.
\end{equation}
We finish this proof.
\end{proof}

\begin{lemma}\label{lemma6}
We choose $\mathcal{H}_{\rm{glo}}^{j} h$ in \myref{equHwhole}, $\mathcal{G}_{\rm{glo}}^{j} g$ in \myref{equglobalu}, $\mathcal{H}_{{\rm cem}}^{j,m} h$ in \myref{equHcem}, and  $\mathcal{G}_{{\rm cem}}^{j,m} h$ in \myref{equGcem},  then
\begin{equation}
    \left \| \left ( \mathcal{H}_{\rm{glo}}- \mathcal{H}_{{\rm cem}} \right )h\right \|_{a}^{2}+\left \| \left ( \mathcal{H}_{\rm{glo}}- \mathcal{H}_{{\rm cem}} \right )h\right \|_{s}^{2}\leq m^{d}\delta ^{2m}\left \| h \right \|_{a}^{2},
\end{equation}
and
\begin{equation}
    \left \| \left ( \mathcal{G}_{\rm{glo}}- \mathcal{G}_{{\rm cem}} \right )g\right \|_{a}^{2}+\left \| \left ( \mathcal{G}_{\rm{glo}}- \mathcal{G}_{{\rm cem}} \right )g\right \|_{s}^{2}\leq \frac{m^{d}\delta ^{2m}}{A}\left \| g \right \|_{L\left ( \Omega \cap \Gamma _{\rm{b}} \right )}^{2},
\end{equation}
where $A=\underset{x\neq 0}{\rm{inf}}\frac{\left \| x \right \|_{a}}{\left \| x \right \|_{L^{2}\left ( \Omega \cap \Gamma _{\rm{b}} \right )}}.$
\end{lemma}
\begin{proof}
    Note that
\begin{equation*}
\begin{aligned}
  a\left(\mathcal{H}_{\rm{glo}}^{j} h, \mathcal{H}_{\rm{glo}}^{j} h\right)+s\left(\pi \mathcal{H}_{\rm{glo}}^{j} h, \lambda\right) &=\int_{K_{j}}  \sigma \left ( h\right ):\epsilon\left ( \mathcal{H}_{\rm{glo}}^{j} h\right )\text{d} x,\\
  s\left ( \mathcal{H}_{\rm{glo}}^{j}h,z_{2} \right )&=\int_{K_{j}}\sigma \left ( h\right ):\epsilon\left ( z_{2}\right )\text{d} x; \\
\end{aligned}
\end{equation*}
we obtain that
\begin{equation*}
    \left \| \mathcal{H}_{\rm{glo}}^{j}h \right \|_{a}^{2}+\left \| \mathcal{H}_{\rm{glo}}^{j}h \right \|_{s}^{2}\leq \left \| \mathcal{H}_{\rm{glo}}^{j}h \right \|_{a\left ( K_j \right )}\left \| h \right \|_{a\left ( K_j \right )},
\end{equation*}
that is
$$\left \| \mathcal{H}_{\rm{glo}}^{j}h \right \|_{a\left ( K_j \right )}\leq \left \| \mathcal{H}_{\rm{glo}}^{j}h \right \|_{a}\leq \left \| h \right \|_{a\left ( K_j \right )}.$$
This yields
\begin{equation*}
    \sum _{K_j\in \mathcal{T}_{H}}\int _{K_j}\sigma \left ( h \right ):\varepsilon \left ( \mathcal{H}_{\rm{glo}}^{j}h \right ) \,\textnormal{d} x\leq \sum _{K_j\in \mathcal{T}_{H}}\left \| h \right \|_{a\left ( K_j \right )}^{2}=\left \| h \right \|_{a}^{2}.
\end{equation*}
 Note that
 \begin{equation*}
\begin{aligned}
a\left(\mathcal{G}_{\rm{glo}}^{j} g, \mathcal{G}_{\rm{glo}}^{j} g\right)+s\left(\pi \mathcal{G}_{\rm{glo}}^{j} g, \lambda\right) &=\int_{\partial K_{j} \cap \Gamma_{\mathrm{b}}} g \cdot \mathcal{G}_{\rm{glo}}^{j} g \mathrm{~d} \sigma,\\
s\left ( \mathcal{G}_{\rm{glo}}^{j}g,z_{2} \right )&=\int_{\partial K_{j} \cap \Gamma_{\mathrm{b}}} g \cdot z_{2} \mathrm{~d} \sigma,
\end{aligned}
\end{equation*}
we obtain $\left \| \mathcal{G}_{\rm{glo}}^{j}h \right \|_{a}\leq \frac{1}{A^{\frac{1}{2}}}\left \| g \right \|_{L^{2}\left ( \partial K_j\cap \Gamma _{\rm{b}} \right )}.$
This yields
\begin{equation*}
    \begin{aligned}
    \sum _{K_j\in \partial K_{j} \cap \Gamma_{\mathrm{b}}}\int_{\partial K_{j} \cap \Gamma_{\mathrm{b}}} g \cdot \mathcal{G}_{\rm{glo}}^{j} g \mathrm{~d} \sigma  &\leq \sum _{K_j\in \partial K_{j} \cap \Gamma_{\mathrm{b}}}\left \| g \right \|_{L^{2}\left ( \partial K_{j} \cap \Gamma_{\mathrm{b}} \right )}\left \| \mathcal{G}_{\rm{glo}}^{j} g  \right \|_{L^{2}\left ( \Gamma _{\rm{b}} \right )}\\
    &\leq \sum _{K_j\in \partial K_{j} }\frac{1}{A}\left \| g \right \|_{L^{2}\left ( \partial K_{j} \cap \Gamma_{\mathrm{b}} \right )}^{2}=\frac{1}{A}\left \| g \right \|_{L^{2}\left ( \Gamma_{\mathrm{b}} \right )}^{2}.
    \end{aligned}
\end{equation*}
\end{proof}
{\rm Finally, we give estimates of the error between the true value and numerical solutions.}

\begin{theorem}
Use the same assumptions and description of symbols as Lemma \ref{lemma1}-Lemma \ref{lemma6}, then we have
\begin{equation}
\begin{aligned}
    \left \| l_{{\rm cem}}^m-\mathcal{H}_{{\rm cem}}^m h+\mathcal{G}_{{\rm cem}}^mg-\tilde{u} \right \|_{a}
    &\leq \Lambda^{-\frac{1}{2}}\left \| f \right \|_{L^{2}\left ( \Omega  \right )}+m^{d}\delta ^{2m}\left(\left \| h \right \|_{a}^{2}+\frac{1}{A}\left \| g \right \|_{L\left ( \Omega \cap \Gamma _{\rm{b}} \right )}^{2}\right)\\
    &\quad +m^{d}\delta^{2m}\left(\sum_{K\in \mathcal{T}_{H}}\left \| l_{\rm{glo}}\right \|_{a}^{2}+\sum_{K\in \mathcal{T}_{H}}\left \|\pi l_{\rm{glo}}\right \|_{s}^{2}\right).
\end{aligned}
\end{equation}
\end{theorem}
\begin{proof}
Note that for ${l}'_{{\rm cem}}\in V_{\rm{glo}},$ we have
$$a\left ( l_{{\rm cem}}^m-\mathcal{H}_{{\rm cem}}^m h+\mathcal{G}_{{\rm cem}}^m g-\tilde{u}, {l}'_{{\rm cem}}\right )=\int _{\Omega }\sigma \left ( l_{{\rm cem}}^m-\mathcal{H}_{{\rm cem}}^m h+\mathcal{G}_{{\rm cem}}^m g-\tilde{u} \right ):\varepsilon \left ( {l}'_{{\rm cem}} \right )=0,$$
and
\begin{equation*}
    \begin{aligned}
    &\quad \left \| l_{{\rm cem}}^m-\mathcal{H}_{{\rm cem}}^m h+\mathcal{G}_{{\rm cem}}^m g-\tilde{u} \right \|_{a}\\
    &=\left \| l_{{\rm cem}}^m-\mathcal{H}_{{\rm cem}}^m h+\mathcal{G}_{{\rm cem}}^m g-\tilde{u}-l_{{\rm cem}}^m+l_{{\rm cem}}^m- {l}'_{{\rm cem}}+{l}'_{{\rm cem}}\right \|_{a}\\
    &\leq\left \| l_{{\rm cem}}^m-\mathcal{H}_{{\rm cem}}^m h+\mathcal{G}_{{\rm cem}}^m g-\tilde{u}-l_{{\rm cem}}^m+{l}'_{{\rm cem}}\right \|_{a}-\left \| l_{{\rm cem}}^m- {l}'_{{\rm cem}} \right \|_{a}\\
    &\leq \left \| {l}'_{{\rm cem}}-\mathcal{H}_{{\rm cem}}^m h+\mathcal{G}_{{\rm cem}}^m g-\tilde{u}\right \|_{a}\\
    &\leq \left \| l_{\rm{glo}}-\mathcal{H}_{\rm{glo}}h+\mathcal{G}_{\rm{glo}}g-\tilde{u}+\left ( \mathcal{H}_{\rm{glo}}h- \mathcal{H}_{{\rm cem}}^m h\right )+\left ( \mathcal{G}_{{\rm cem}}^m g-\mathcal{G}_{\rm{glo}}g \right )-l_{\rm{glo}}+{l}'_{{\rm cem}}\right \|_{a}\\
    & \leq \left \| l_{\rm{glo}}-\mathcal{H}_{\rm{glo}}h+\mathcal{G}_{\rm{glo}}g-\tilde{u} \right \|_{a}+\left \|  \left ( \mathcal{H}_{\rm{glo}}h- \mathcal{H}_{{\rm cem}}^m h\right )\right \|_{a}\\
    &\quad +\left \| \left ( \mathcal{G}_{{\rm cem}}^m g-\mathcal{G}_{\rm{glo}}g \right ) \right \|_{a}+\left \| l_{\rm{glo}}-{l}'_{{\rm cem}}\right \|_{a}\\
    &\leq \left ( \Lambda \right )^{-\frac{1}{2}}\left \| f \right \|_{L^{2}\left ( \Omega  \right )}+m^{d}\delta ^{2m}\left \| h \right \|_{a}^{2}+\frac{m^{d}\delta ^{2m}}{A}\left \| g \right \|_{L\left ( \Omega \cap \Gamma _{\rm{b}} \right )}^{2}+\left \| l_{\rm{glo}}-{l}'_{{\rm cem}} \right \|_{a}.
    \end{aligned}\rm{glo}
\end{equation*}
Then it is enough to estimate $\left \| l_{\rm{glo}}-{l}'_{{\rm cem}} \right \|_{a}.$
Combining the Lemma \ref{lemma5}, we have
\begin{equation*}
\left \| l_{\rm{glo}}-{l}'_{{\rm cem}} \right \|_{a}
\leq m^{d}\delta^{2m}\left(\sum_{K\in \mathcal{T}_{H}}\left \| l_{\rm{glo}}\right \|_{a}^{2}+\sum_{K\in \mathcal{T}_{H}}\left \|\pi l_{\rm{glo}}\right \|_{s}^{2}\right).
\end{equation*}
This proof this theorem.
\end{proof}
{\rm Moreover, from the proof of \autoref{thm:111}  and \myref{equ:lwhole}, we have }
\begin{equation*}
\begin{aligned}
\left\| l_{\rm{glo}}\right\|_{a}+\left\|\pi l_{\rm{glo}}\right\|_{s} &=\left\| l_{\rm{glo}}\right\|_{a}+\left\|\pi\left(\tilde{u}+\mathcal{H}_{\text {cem }} h-\mathcal{G}_{\rm{glo}} g\right)\right\|_{s}\\
&\leq \left\| l_{\rm{glo}}\right\|_{a}+\left\|\pi \tilde{u}\right\|_{s}+\left\|\pi \mathcal{H}_{\rm{glo}} h\right\|_{s}+\left\|\pi \mathcal{H}_{\rm{glo}} g\right\|_{s} \\
& \leq O\left(H^{-1}\right)+O\left(1\right).
\end{aligned}
\end{equation*}

\section{Numerical Results}\label{sub:4}
{\rm In this section, we present three different model problems and the corresponding numerical results and discuss these three problems respectively, to emphasize that the boundary corrector proposed in this paper can maintain accuracy at high contrast factor settings. The high-contrast model where we choose three fixed-value Young’s modulus $E\left(x\right)$ is shown in \autoref{fig:Medium config 1}, \autoref{fig:Medium config 2} and \autoref{fig:Medium config 3}  for $2D$ to verify the corrective effects of  Dirichlet correctors, Neumann  correctors and the combination of these two correctors, respectively.
We give the definition of $\lambda\left(x\right)$ and $\mu\left(x\right)$ as following:
\begin{equation*}
\begin{aligned}
    \lambda\left(x\right)&=\frac{\nu}{\left(1+\nu\right)\left(1-2 \nu\right)} E\left(x\right),\\ \mu\left(x\right)&=\frac{1}{2\left(1+\nu\right)} E\left(x\right),
\end{aligned}
\end{equation*}
where the Poisson’s ratio $\nu$ is $0.25$.  We let $u_{h}$ denote the fine scale reference given by \myref{equ:galerkin} and let $u_{{\rm cem}}$ denote the CEM-GMsFEM approximation given by \myref{equcemu}.
In the following experiments, we set $D =\left [0,1  \right ]$.

In order to clearly state experimental results , \autoref{fig:symbol} gives some notation as follows.
\begin{table}[H]
\setlength{\abovecaptionskip}{0pt}
\setlength{\belowcaptionskip}{10pt}
\centering
\topcaption{Simplified description of symbols.}
\label{fig:symbol}
\begin{tabular}{l|llllll}
\cline{1-2}
Parameters                                        & Symbols &  &  &  &  &  \\ \cline{1-2}
Number of oversampling coarses                    & $Noc$      &  &  &  &  &  \\
Number of basis functions in every coarse element & $Nbf $    &  &  &  &  &  \\
Length of every coarse element size               & $L $      &  &  &  &  &  \\
Contrast value                                    & $E$      &  &  &  &  &  \\ \cline{1-2}
\end{tabular}
\end{table}


\subsection{Model problem 1}\label{sub:model1}
In the first model problem, we take the configuration of Young's modulus $E\left(x_1, x_2\right)$ and Poisson's ratio $\nu\left(x_1, x_2\right)$ as \autoref{fig:Medium config 1}. We fix $E=1.0$ and $\nu=0.25$ for the matrix. For inclusions, we choose $E$ either $10^4$, $10^5$ or $10^6$, while $\nu=0.45$ for all experiments. We emphasize that high contrast of $E$ determines the main difficulty of the problem. The Dirichlet BCs are set as
\begin{equation}
u\left(x_1, x_2\right)=h\left(x_1, x_2\right)=\left(\begin{aligned}
&x_1 + \exp\left(x_1x_2\right) \\
&\cos\left(x_1\right)\cos\left(x_2\right)
\end{aligned}\right)\label{model1}
\end{equation}
on $\partial \Omega$. and the source function $f=\left(f^1, f^2\right)$, where
\[f^1\left(x_1,x_2\right) = \begin{cases}
1.0, & \text{ if } \left(x_1, x_2\right) \in \left(\frac{1}{8}, \frac{7}{8}\right) \times \left(\frac{3}{8}, \frac{5}{8}\right) \cup \left(\frac{3}{8}, \frac{5}{8}\right) \times \left(\frac{1}{8}, \frac{7}{8}\right), \\0.0, & \text{ otherwise},
\end{cases}\]
and $f^2\left(x_1, x_2\right)=0$ for all $\left(x_1,x_2\right)\in D$.

\begin{figure}[h]
\centering
\includegraphics[width=0.5\textwidth]{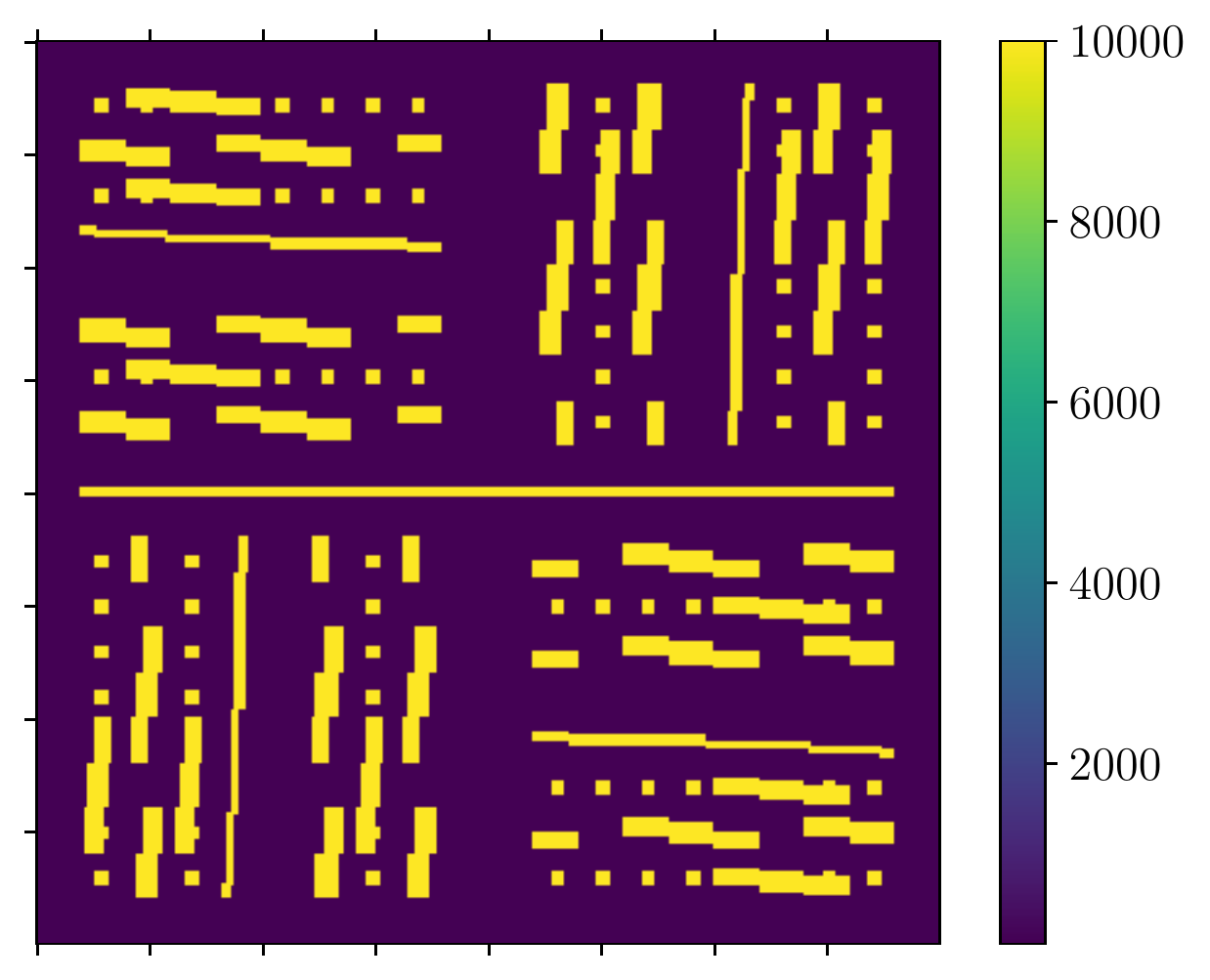}
\caption{The medium configure of model problem 1.} \label{fig:Medium config 1}
\end{figure}

We first check the decay of $\mathcal{H}^m_\textnormal{cem}h$ with respect to different oversampling layers $m$. We set the coarse grid as $64\times 64$, and $\dim \left(V_i^{\rm{aux}}\right)=3$ for all $i$.
We first examine the exponential convergence of $\left(\mathcal{H}_{\rm cem}^{m}-\mathcal{H}_{\rm{glo}}\right) h$ by setting $L=1 / 20$ and $Nbf=2$. The results are reported in \autoref{table1}, where we introduce notations
\begin{equation*}
\left\|u_{h}-u_{\mathrm{cem}}\right\|_{H^{1}\left(\Omega\right)}^{\mathrm{rel-H}}:=\frac{\left\|\mathcal{H}^m_{\textnormal{cem}}h-\mathcal{H}_{\rm{glo}}h\right\|_{a}} {\left\|\mathcal{H}_{\rm{glo}h}\right\|_{a}}
\end{equation*}
and
\begin{equation*}
\quad \left\|u_{h}-u_{\mathrm{cem}}\right\|_{L^{2}\left(\Omega\right)}^{\mathrm{rel-H}}:=\frac{\left\|\mathcal{H}^m_\textnormal{cem}h-\mathcal{H}_{\rm{glo}}h\right\|_{L^{2}(\Omega)}}{\| \mathcal{H}_{\rm{glo}}h \|_{L^{2}\left(\Omega\right)}}
\end{equation*}
to measure errors.

As shown by left of \autoref{fig:errors1}, we can see that the decay rates of $\mathcal{H}^m_{\rm cem}h-\mathcal{H}_{\rm glo}h$ are almost indistinguishable with respect to different $E\in \{10^4, 10^5, 10^6\}$. We could observe an exponential decay pattern, which convinces the estimate in Lemma \ref{lemma6}.


We check the numerical errors of our methods by right of  \autoref{fig:errors1}. As the number of oversampling layers increases from $1$ to $3$, we see the same independent relationship between error and contrast, but as the number of oversampling layers increases to $4$, the error decreases as the contrast increases from $10^4$ to $10^6$, indicating that our method is very effective and accurate in a high contrast model.


\begin{figure}[h]
\centering
\includegraphics[width=1.0\textwidth]{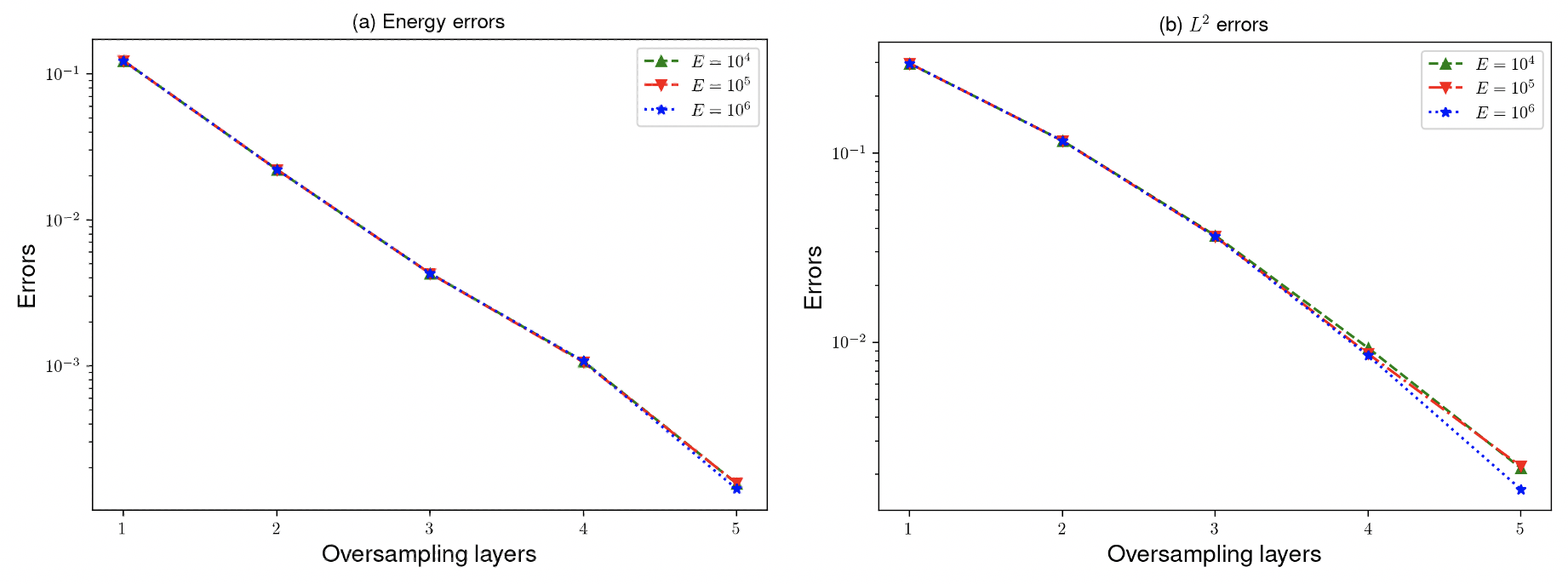}
\caption{(a) Energy norm of relative errors between $\mathcal{H}^m_\textnormal{cem}h$ and  $\mathcal{H}_{\rm{glo}}h$ with respect to different oversampling layers $m$ and Young's modulus $E$, (b) $L^2$ norm of relative errors  between $\mathcal{H}^m_\textnormal{cem}h$ and  $\mathcal{H}_{\rm{glo}}h$ with respect to different oversampling layers $m$ and Young's modulus $E$.} \label{fig:errors1}
\end{figure}



\begin{table}[H]
\centering
\topcaption{The Dirichlet correctors of the model problem \myref{model1} in the energy $H^1$ and $L^{2}$ norm with different $Noc$ and different $E $, while $Nbf=2$ and $L=1/20.$}
\label{table1}
\begin{tabular}{|c|c|c|c|}
\hline
$E$                    & $Noc$ & $\left\|u_{h}-u_{\mathrm{cem}}\right\|_{H^{1}\left(\Omega\right)}^{\mathrm{rel-H}} $ & $\left\|u_{h}-u_{\mathrm{cem}}\right\|_{L^{2}\left(\Omega\right)}^{\mathrm{rel-H}} $ \\ \hline
\multirow{6}{*}{$10^4$} &1 & 12.31\% & 29.66\% \\ \cline{2-4}
     & 2   & 2.22\% & 11.60\% \\ \cline{2-4}
     & 3   & 0.43\%  & 3.66\% \\ \cline{2-4}
     & 4   & 0.11\%  &  0.93\%  \\ \cline{2-4}
     & 5   & 0.02\% &  0.22\% \\ \cline{2-4}    
     & Ref   & 117.02\% &  68.38\%  \\  \hline
\multirow{6}{*}{$10^5$}& 1   & 12.31\% & 29.67\% \\ \cline{2-4}
      & 2   & 2.22\% & 11.59\%\\ \cline{2-4}
      & 3   & 0.43\% & 3.61\%  \\ \cline{2-4}
      & 4   & 0.11\%  & 0.87\%   \\ \cline{2-4}
      & 5   & 0.02\%  &  0.22\% \\ \cline{2-4}
     & Ref   & 37.01\% &  21.62\%   \\  \hline
\multirow{6}{*}{$10^6$}& 1   & 12.31\% & 29.68\% \\ \cline{2-4}
     & 2   & 2.22\% & 11.59\% \\ \cline{2-4}
     & 3   & 0.43\% & 0.361\%  \\ \cline{2-4}
     & 4   & 0.11\%  &  0.85\%  \\ \cline{2-4}
     & 5   & 0.02\%  &  0.17\% \\ \cline{2-4}
     & Ref   & 11.70\% &  6.83\%  \\ \hline
\end{tabular}
\end{table}

\subsection{Model problem 2}\label{5.2}

In this subsection, we study the following inhomogeneous Neumann BVP:
\begin{equation}
\left\{\begin{array}{lr}
u\left(x_{1}, x_{2}\right)=\left (0.0,0.0\right )^T, & \forall\left(x_{1}, x_{2}\right) \in\left(0,1\right) \times\{1\}, \\
 \sigma\left(u\left ( x_{1} ,x_{2}\right )\right)\cdot n=g\left(x_{1}, x_{2}\right)=\left (-1.0,0.0\right )^T, & \forall\left(x_{1}, x_{2}\right) \in\{0\} \times\left(0,1\right), \\
 \sigma\left(u\left ( x_{1} ,x_{2}\right )\right)\cdot n=g\left(x_{1}, x_{2}\right)=\left (1.0,0.0\right )^T, & \forall\left(x_{1}, x_{2}\right) \in\{1\} \times\left(0,1\right) ,\\
 \sigma\left(u\left ( x_{1} ,x_{2}\right )\right)\cdot n=g\left(x_{1}, x_{2}\right)=\left (1.0,0.0\right )^T, & \forall\left(x_{1}, x_{2}\right) \in\left(0,0.5\right) \times\{0\}, \\
 \sigma\left(u\left ( x_{1} ,x_{2}\right )\right)\cdot n=g\left(x_{1}, x_{2}\right)=\left (0.0,0.0\right )^T, & \forall\left(x_{1}, x_{2}\right) \in\left(0.5,1\right) \times\{0\}.
\end{array}\right.\label{equ:model2}
\end{equation}

In the second model problem, we take the configuration of Young’s modulus $E\left(x_1, x_2\right)$ and Poisson’s ratio
$\nu\left(x_1, x_2\right)$ as \autoref{fig:Medium config 2}.

\begin{figure}[h]
\centering
\includegraphics[width=0.5\textwidth]{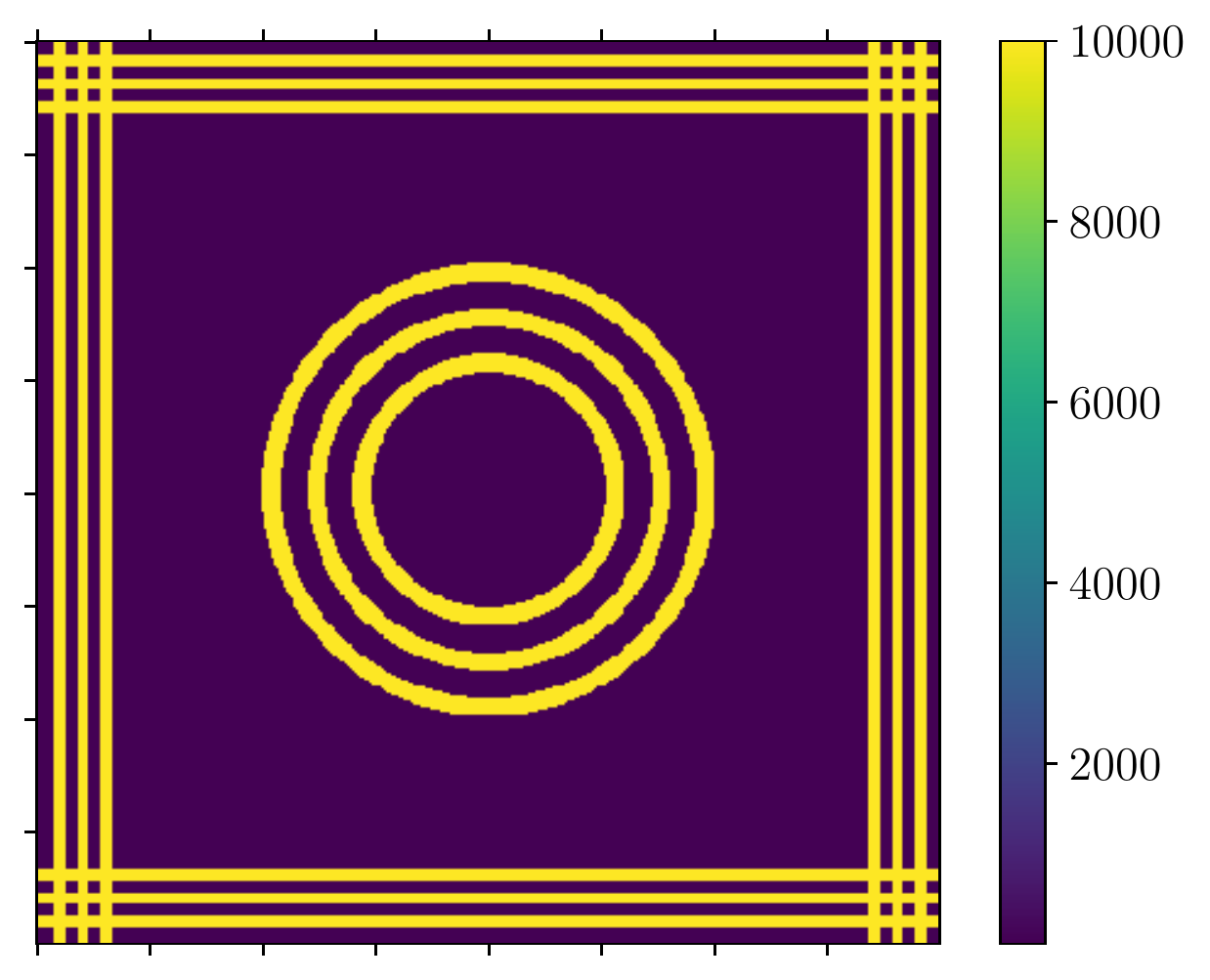}
\caption{The medium configure of model problem 2.} \label{fig:Medium config 2}
\end{figure}

We check the decay of $\mathcal{G}^m_\textnormal{cem}g$ with respect to different oversampling layers $m$. We set the coarse grid as $64\times 64$, and $\dim \left(V_{i}^{\rm aux}\right)=3$ for all $i$.
We examine the exponential convergence of $\left(\mathcal{G}^{m}-\mathcal{G}^{\text {glo}}\right) g$ by setting $L=1/20$ and $Nbf=2$. The results are reported in \autoref{table2}, where we introduce notations
\begin{equation*}
\left\|u_{h}-u_{\mathrm{cem}}\right\|_{H^{1}\left(\Omega\right)}^{\mathrm{rel-G}}:=\frac{\left\|\mathcal{G}^m_\textnormal{cem}g-\mathcal{G}_{\rm{glo}}g\right\|_{a}}{\|\mathcal{G}_{\rm{glo}}g\|_{a}}
\end{equation*}
and
\begin{equation*}
 \left\|u_{h}-u_{\mathrm{cem}}\right\|_{L^{2}\left(\Omega\right)}^{\mathrm{rel-G}}:=\frac{\left\|\mathcal{G}^m_\textnormal{cem}g-\mathcal{G}_{\rm{glo}}g\right\|_{L^{2}\left(\Omega\right)}}{\| \mathcal{G}_{\rm{glo}}g \|_{L^{2}\left(\Omega\right)}}
 \end{equation*}
to measure errors.

\begin{figure}[h]
\centering
\includegraphics[width=1.0\textwidth]{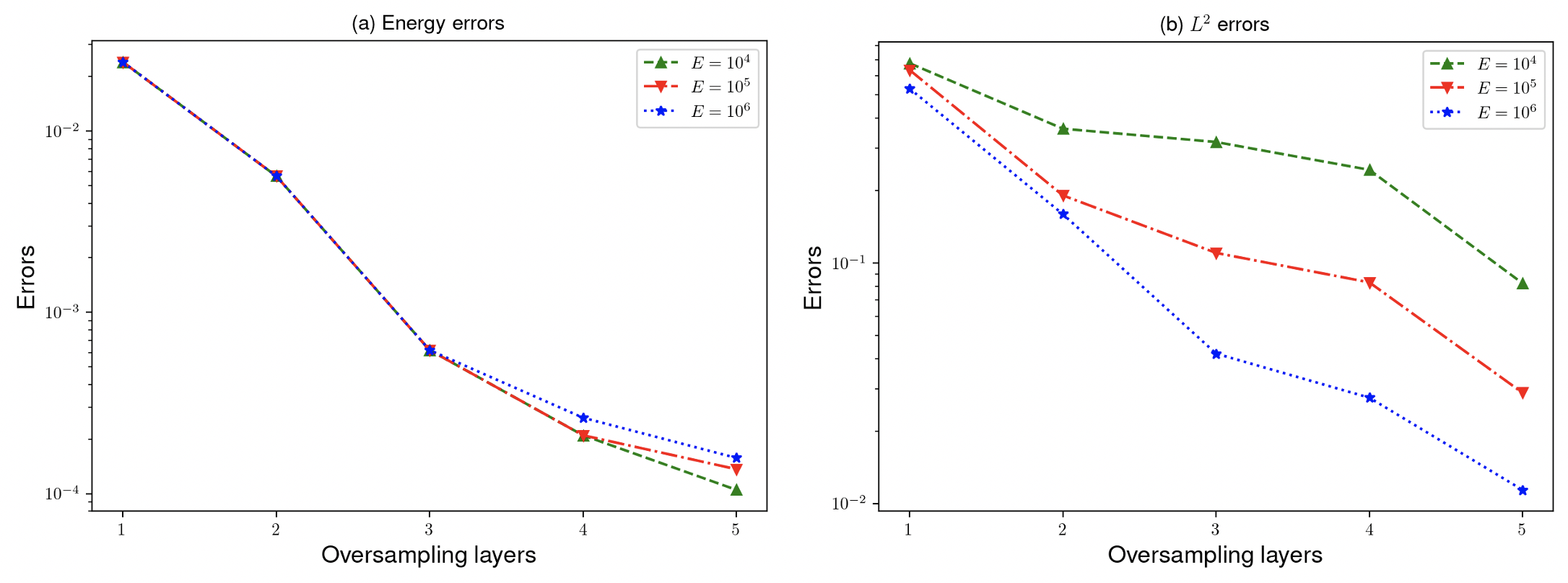}
\caption{(a) Energy norm of relative errors between $\mathcal{G}^m_\textnormal{cem}g$ and  $\mathcal{G}_{\rm{glo}}g$ with respect to different oversampling layers $m$ and Young's modulus $E$, (b) $L^2$ norm of relative errors between $\mathcal{G}^m_\textnormal{cem}g$ and  $\mathcal{G}_{\rm{glo}}g$ with respect to different oversampling layers $m$ and Young's modulus $E$} \label{fig:errors3}
\end{figure}



This experiment serves the similar purpose as the experiment in Subsection \ref{sub:model1} and it further validates that the accuracy of our proposed method increases when the oversampling region is expanded. In this experiment, we let $H = 1/80 $ and $Nbf= 2$, and similar error proficiency results to those in \autoref{table2} can be observed in the results in \autoref{table1}. We visualize the results of \autoref{table2} in \autoref{fig:errors3}. This test shows that the accuracy of our proposed method increases as the contrast increases and also verifies the applicability of our method to high contrast materials.

\begin{table}[H]
\centering
\topcaption{The Neumann correctors of the model problem \myref{equ:model2} in the energy $H^1$ and $L^{2}$ norm with different $Noc$ and different $E $, while $Nbf=2$ and $H=1/20.$}
\label{table2}
\begin{tabular}{|c|c|c|c|}
\hline
$E$                    & $Noc$ & $\left\|u_{h}-u_{\mathrm{cem}}\right\|_{H^{1}\left(\Omega\right)}^{\mathrm{rel-G}} $ & $\left\|u_{h}-u_{\mathrm{cem}}\right\|_{L^{2}\left(\Omega\right)}^{\mathrm{rel-G}}$ \\ \hline
\multirow{6}{*}{$10^4$} &1 & 2.40\% & 67.54\% \\ \cline{2-4}
     & 2   & 0.56\% &  36.10\% \\ \cline{2-4}
     & 3   & 0.06\%  & 31.79\%   \\ \cline{2-4}
     & 4   & $<1.00\times10^{-3}$  &  24.36\%  \\ \cline{2-4}
     & 5   & $<1.00\times10^{-3}$    & 8.22\%     \\ \cline{2-4}
     & Ref   & 57.25\%  & 9.53\%   \\ \hline
\multirow{6}{*}{$10^5$} & 1   & 2.40\% &  63.28\%\\ \cline{2-4}
      & 2   & 0.56\% & 19.07\%  \\ \cline{2-4}
      & 3   & 0.06\% &  11.00\%   \\ \cline{2-4}
      & 4   & $<1.00\times10^{-3}$  &  8.30\%  \\ \cline{2-4}
      & 5   & $<1.00\times10^{-3}$   &  2.88\%     \\ \cline{2-4}
     & Ref   & 54.29\%  & 8.26\%  \\ \hline
\multirow{6}{*}{$10^6$} & 1   & 2.40\% & 52.99\% \\ \cline{2-4}
     & 2   & 0.56\% &  15.94\%  \\ \cline{2-4}
     & 3   & 0.06\% &  4.20\% \\ \cline{2-4}
     & 4   & $<1.00\times10^{-3}$  & 2.75\%  \\ \cline{2-4}
     & 5   & $<1.00\times10^{-3}$    &    1.14\%   \\ \cline{2-4}
     & Ref   & 53.97\%  & 6.22\%   \\ \hline
\end{tabular}
\end{table}


\subsection{Model problem 3}
In this subsection, we consider a more complex inhomogeneous Neumann BVP than \myref{equ:model2} in \autoref{fig:Medium config 3}.

\begin{figure}[h]
\centering
\includegraphics[width=0.5\textwidth]{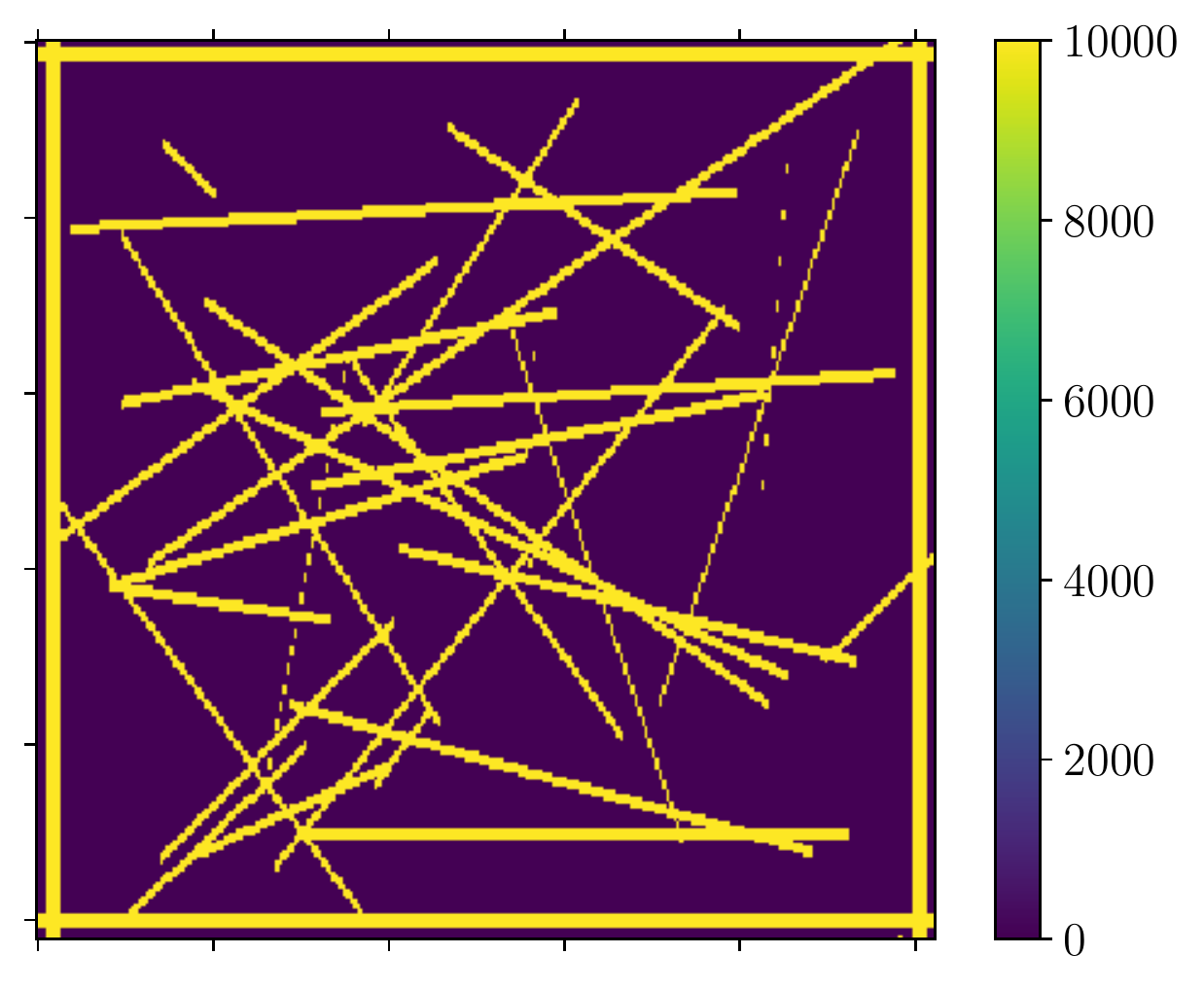}
\caption{The medium configure of model problem 3.} \label{fig:Medium config 3}
\end{figure}

Two experiments are carried out in this part, one to confirm the validity of the Neumann correctors and the other to see how effectively the Neumann and Dirichlet correctors work together to enhance our numerical findings. We introduce notations
\begin{equation*}
\left\|u_{h}-u_{\mathrm{cem}}\right\|_{H^{1}\left(\Omega\right)}:=\frac{\left\| l_{\rm{cem}}^{m}-\mathcal{H}^m_\textnormal{cem}h-\tilde{u} \right\|_{a}} {\|\mathcal{H}_{\rm{glo}}h\|_{a}}
\end{equation*}
and
\begin{equation*}
\quad \left\|u_{h}-u_{\mathrm{cem}}\right\|_{L^{2}\left(\Omega\right)}:=\frac{\left\|l_{\rm{cem}}^{m}-\mathcal{H}^m_\textnormal{cem}h-\tilde{u} \right\|_{L^{2}\left(\Omega\right)}}{\| \mathcal{H}_{\rm{glo}}h \|_{L^{2}\left(\Omega\right)}}
\end{equation*}
to measure errors under $E=10^4$  and $Nbf=2$. The results are reported in \autoref{table4} and \autoref{table5}. We visualize the results of \autoref{table4} in \autoref{fig:errors8}, and \autoref{table5} in \autoref{fig:errors6} and \autoref{fig:errors5}.

\begin{figure}[h]
\centering
\includegraphics[width=1.0\textwidth]{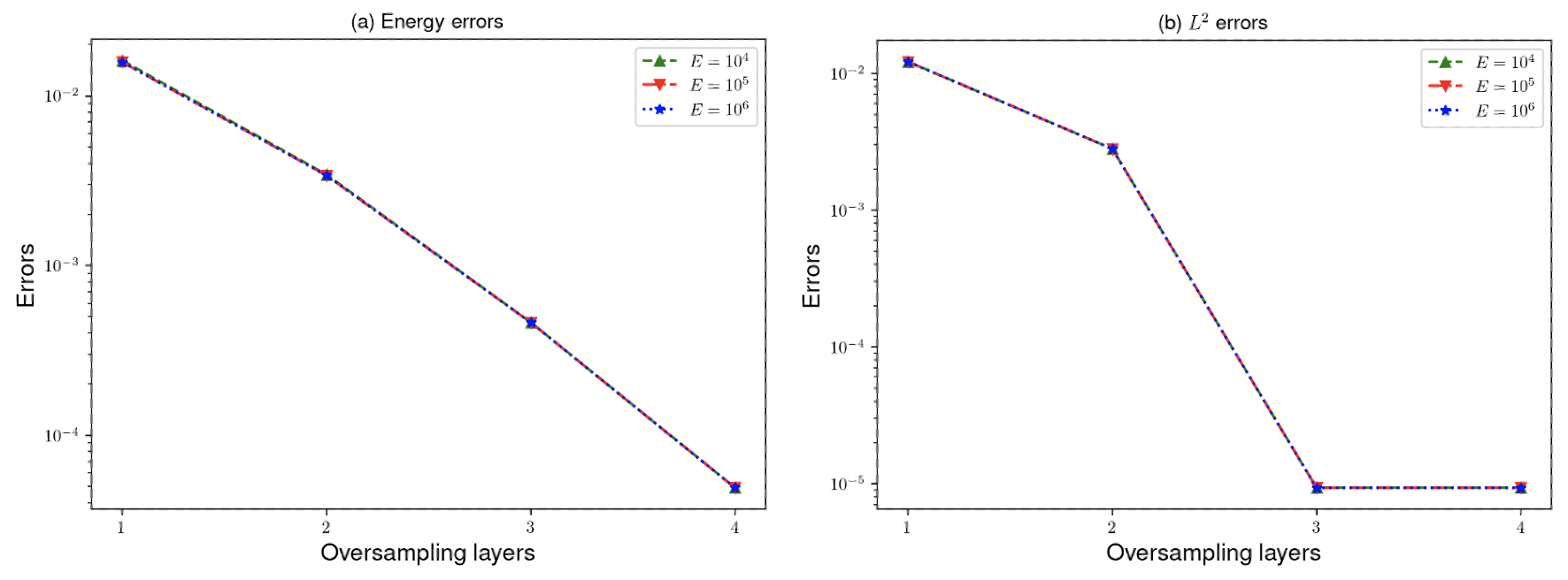}
\caption{(a) Energy norm of relative errors between $\mathcal{G}^m_\textnormal{cem}g$ and  $\mathcal{G}_{\rm{glo}}g$ with respect to different oversampling layers $m$ and Young's modulus $E$, (b) $L^2$ norm of relative errors between $\mathcal{G}^m_\textnormal{cem}g$ and  $\mathcal{G}_{\rm{glo}}g$ with respect to different oversampling layers $m$ and Young's modulus $E$.} \label{fig:errors8}
\end{figure}



\begin{table}[H]
\centering
\topcaption{The Neumann correctors of the model problem in \autoref{fig:Medium config 3} in the energy $H^1$ and $L^{2}$ norm with different $Noc$ and different $E $, while $Nbf=2$ and $H=1/20.$}
\label{table4}
\begin{tabular}{|c|c|c|c|}
\hline
$E$   & $Noc$  & $\left\|u_{h}-u_{\mathrm{cem}}\right\|_{H^{1}\left(\Omega\right)}^{\mathrm{rel-G}} $ & $\left\|u_{h}-u_{\mathrm{cem}}\right\|_{L^{2}\left(\Omega\right)}^{\mathrm{rel-G}}$ \\ \hline
\multirow{6}{*}{$10^4$} &1 & 0.23\% &$<1.00\times10^{-4}$\\ \cline{2-4}
     & 2   & 0.05\% &  $<1.00\times10^{-5}$ \\ \cline{2-4}
     & 3   & $<1.00\times10^{-4}$  & 0.00\%   \\ \cline{2-4}
     & 4   & $<1.00\times10^{-4}$  &  0.00\%  \\ \cline{2-4}
     & ref   & 14.26\%    & 0.11\%     \\ \hline
\multirow{6}{*}{$10^5$} & 1   & 0.23\% &  $<1.00\times10^{-4}$\\ \cline{2-4}
      & 2   & 0.05\% & $<1.00\times10^{-5}$  \\ \cline{2-4}
      & 3   & $<1.00\times10^{-4}$ &  0.00\%   \\ \cline{2-4}
      & 4   & $<1.00\times10^{-4}$  &  0.00\%  \\ \cline{2-4}
      & ref   & 14.29\%   &  0.11\%      \\ \hline
\multirow{6}{*}{$10^6$} &1   &0.23\% & $<1.00\times10^{-4}$ \\ \cline{2-4}
     & 2   & 0.05\% & $<1.00\times10^{-5}$ \\ \cline{2-4}
     & 3   & $<1.00\times10^{-4}$ &  0.00\% \\ \cline{2-4}
     & 4   & $<1.00\times10^{-4}$  & 0.00\% \\ \cline{2-4}
     & ref   & 14.29\%    &    0.11\%     \\ \hline
\end{tabular}
\end{table}

In first experiment, as shown in \autoref{fig:errors8},  we discovered that for fixed $Nbf=2$ and $H=1/20$, the error results remained very low even as the contrast increased, demonstrating that our method is successful in achieving efficient solutions to high contrast problems under Neumann correctors,  and these findings corroborate the Neumann BVPs discussed in Subsection \ref{5.2}.

\begin{figure}[h]
\centering
\includegraphics[width=0.5\textwidth]{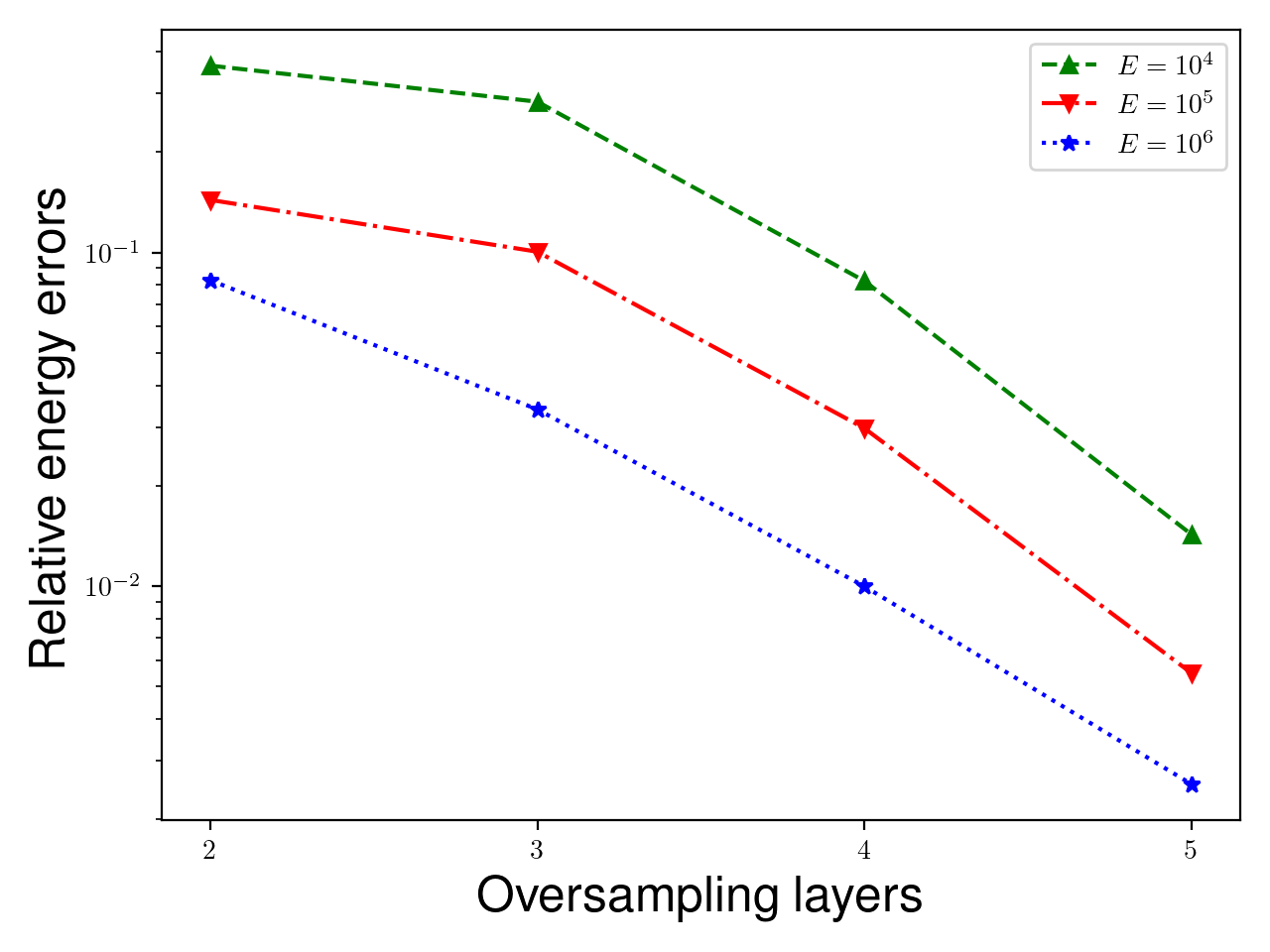}
\caption{Energy norm of relative errors of numerical solutions of Neumann BVPs with respect to different oversampling layers $m$ and Young's modulus $E$. } \label{fig:errors6}
\end{figure}

\begin{figure}[h]
\centering
\includegraphics[width=0.5\textwidth]{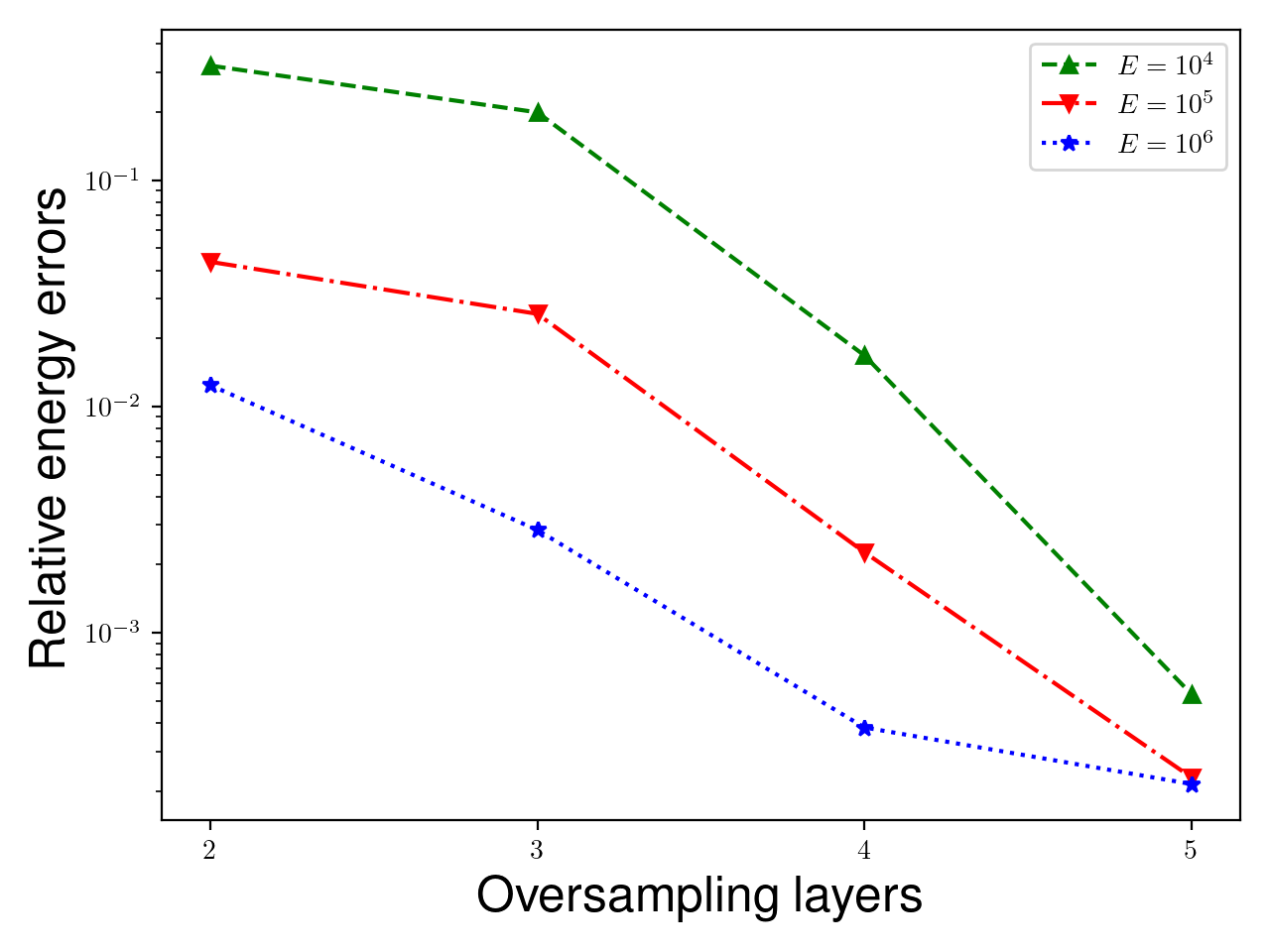}
\caption{$L^2$ norm of relative errors of numerical solutions of Neumann BVPs with respect to different oversampling layers $m$ and Young's modulus $E$. } \label{fig:errors5}
\end{figure}

\begin{table}[H]
\centering
\topcaption{The numerical errors of the model problem in \autoref{fig:Medium config 3} in the energy $H^1$ and $L^{2}$ norm with different $Noc$ and different $E $, while $Nbf=2$ and $H=1/20.$}
\label{table5}
\begin{tabular}{|c|c|c|c|}
\hline
$E$   & $Noc$ & $\left\|u_{h}-u_{\mathrm{cem}}\right\|_{H^{1}\left(\Omega\right)} $ & $\left\|u_{h}-u_{\mathrm{cem}}\right\|_{L^{2}\left(\Omega\right)}$ \\ \hline
\multirow{6}{*}{$10^4$} &2 & 25.45\% & 6.26\% \\ \cline{2-4}
     & 3   & 19.85\% &  3.89\% \\ \cline{2-4}
     & 4   & 5.76\%  & 0.33\%   \\ \cline{2-4}
     & 5   & 1.00\%  &  0.01\%  \\ \cline{2-4}
     & ref   & 79.08\%    & 19.53\%     \\ \hline
\multirow{6}{*}{$10^5$}& 2   & 9.47\% &  0.73\%\\ \cline{2-4}
      & 3   & 6.63\% & 0.43\%  \\ \cline{2-4}
      & 4   & 1.96\% &  0.04\%   \\ \cline{2-4}
      & 5   & 0.36\%  &  $<1.00\times10^{-4}$  \\ \cline{2-4}
      & ref   & 65.87\%   &  16.94\%      \\ \hline
\multirow{6}{*}{$10^6$} &2   &5.40\% & 0.21\% \\ \cline{2-4}
     & 3   & 2.22\% &  0.048\%  \\ \cline{2-4}
     & 4   & 0.06\% &  $<1.00\times10^{-4}$ \\ \cline{2-4}
     & 5   & 0.17\%  & $<1.00\times10^{-4}$ \\ \cline{2-4}
     & ref   & 65.40\%    &    16.75\%     \\ \hline
\end{tabular}
\end{table}

In second experiment, as shown in \autoref{fig:errors6} and \autoref{fig:errors5}, for fixed $Nbf=2$ and $H=1/20$, the inaccuracy improves dramatically as the number of $Noc$ for the same contrast increases. Specially, with the $Noc$ number of $5$, for both the energy and $L^2$ norm of the combination of  Dirichlet and  Neumann correctors drop below $1\%$ for contrast levels of $10^4, 10^5,$ and $10^6.$ Summarizing this data, our technique achieves excellent efficiency for high-contrast materials.
}

\section*{Acknowledgments}
The research of Eric Chung is partially supported by the Hong Kong RGC General Research Fund (Project numbers 14304719 and 14302620) and CUHK Faculty of Science Direct Grant 2021-22.

\small

\bibliographystyle{siam}

\bibliography{cem-gmsfeminhonmo2}

\end{document}